\documentclass[11pt]{amsart}

\usepackage{amssymb,amsfonts,theoremref,mathrsfs}
\usepackage{mdwlist}

\usepackage{newtxtext,newtxmath}

\usepackage{amssymb,textcomp}
\usepackage{enumerate}

\usepackage[utf8]{inputenc}
\usepackage[T1]{fontenc}

\usepackage[mathscr]{eucal}
\usepackage{xcolor}

\usepackage{tikz}
\tikzstyle{startstop} = [rectangle, rounded corners, minimum width=3cm, minimum height=1cm,text centered, draw=black, fill=red!30]

\usepackage[a4paper, twoside=false, vmargin={2cm,3cm}, includehead]{geometry}

\usepackage{comment}
\newtheorem{lemma}{Lemma}[section]
\newtheorem{theorem}[lemma]{Theorem}
\newtheorem*{theorem*}{Theorem}
\newtheorem{corollary}[lemma]{Corollary}

\newtheorem*{claim*}{Claim}

\newtheorem{proposition}[lemma]{Proposition}
\newtheorem*{proposition*}{Proposition}

\newtheorem*{problem*}{Problem}

\theoremstyle{definition}

\DeclareMathOperator{\ord}{ord}
\DeclareMathOperator{\vol}{vol}
\DeclareMathOperator{\lcm}{lcm}

\DeclareMathOperator*{\E}{\mathbb{E}}

\newcommand{\C}{{\mathbb C}}

\newcommand{\N}{{\mathbb N}}
\newcommand{\Q}{{\mathbb Q}}
\newcommand{\R}{{\mathbb R}}

\newcommand{\Z}{{\mathbb Z}}

\DeclareMathOperator{\dist}{dist}

\DeclareMathOperator{\Imag}{Im}
\renewcommand{\Im}{\Imag}

\DeclareMathOperator{\Lip}{Lip}

\DeclareMathOperator{\ML}{ML}

\usepackage[unicode]{hyperref}
\hypersetup{
  colorlinks=true,
  linktocpage=true,
  linkcolor=red,
  filecolor=blue,
  citecolor=blue,
  urlcolor=cyan,
}

\address{St John's College, Oxford and Mathematical Institute, University of Oxford; St Giles', Oxford OX1 3JP, UK}
\email{noah.kravitz@maths.ox.ac.uk}

\address{450 Jane Stanford Way Building 380, Sloan Mathematical Center, Stanford, CA 94305}
\email{jleng01@stanford.edu}

\begin{document}

\title{Quantitative pyjama}
\author{Noah Kravitz and James Leng}

\begin{abstract}
The ``pyjama stripe'' with parameter $\varepsilon>0$ is the set $E(\varepsilon)$ of all complex numbers $z$ such that the distance from $\Re(z)$ to the nearest integer is at most $\varepsilon$.  The Pyjama Problem of Iosevich, Kolountzakis, and Matolcsi asks whether, for every choice of $\varepsilon>0$, it is possible to cover the entire complex plane with finitely many rotations of $E(\varepsilon)$ around the origin.  Manners obtained an affirmative answer to this question by studying a $\times 2, \times 3$-type problem over a suitable solenoid.  Manners's argument provided no quantitative bounds (in terms of $\varepsilon$) on the number of rotations required, and Green has highlighted the problem of obtaining such quantitative bounds.  Our main result is that $\exp\exp\exp(\varepsilon^{-O(1)})$ rotations of $E(\varepsilon)$ suffice to cover the complex plane.  Our analysis makes use of the entropic tools developed by Bourgain, Lindenstrauss, Michel, and Venkatesh for quantitative $\times 2, \times 3$-type results.
\end{abstract}

\maketitle

\section{Introduction}
The topic of this paper is a charming geometric problem introduced by Iosevich, Kolountzakis, and Matolcsi~\cite{iosevich2006covering} in 2006.  For small $\varepsilon>0$, consider the ``pyjama stripe'' $$E(\varepsilon) := \{z \in \mathbb{C}: \Re(z) \in (-\varepsilon, \varepsilon) \pmod{1}\},$$
which consists of width-$2\varepsilon$ vertical stripes centered at the integers.  The Pyjama Problem asks whether, for every choice of $\varepsilon>0$, it is possible to cover the whole complex plane with finitely many rotations of $E(\varepsilon)$ around the origin.  When $\varepsilon>1/3$, it is easy to write down a set of three rotations of $E(\varepsilon)$ whose union covers the entire plane.

In general, each rotation of $E(\varepsilon)$ is a doubly-periodic set.  One natural strategy is to work with rotations that are doubly-periodic with respect to a common lattice $\Lambda$; then, to check that these rotations together cover the whole plane, it suffices to check that they cover a fundamental domain of $\Lambda$.  Malikiosis, Matolcsi, and Ruzsa~\cite{malikiosis2013note} showed that, somewhat surprisingly, there are no such ``periodic'' solutions once $\varepsilon<1/3$.  Nonetheless, using new ideas involving irrational rotations, they confirmed that the Pyjama Problem has an affirmative answer for $\varepsilon>1/5$.  The same authors showed that a finite set of generic rotations will never cover the whole plane; this reduces the viability of naive probabilistic proof strategies based on random rotations.

In 2013, Manners~\cite{Man13} spectacularly proved that the Pyjama Problem has an affirmative answer for all $\varepsilon>0$.  His beautiful work brought topological dynamics into the picture by exploiting an analogy with Furstenberg's $\times 2, \times 3$ Theorem (about $\times 2, \times 3$-invariant closed subsets of $\mathbb{R}/\mathbb{Z}$).  As we will discuss later, Manners showed that if one considers a particular family of rotations, then the Pyjama Problem is a consequence of a $\times 2, \times 3$-type problem over a suitable solenoid.

The drawback of Manners's proof is that, due to its use of compactness arguments, it provides no quantitative upper bounds on how many rotations (as a function of $\varepsilon$) are necessary.  Remedying this situation appears as Problem 41 in Ben Green's open problem list~\cite{green2024100}.  Our main result provides fairly reasonable quantitative bounds.

\begin{theorem}\label{thm:main}
There is a universal constant $C>0$ such that for every $0 <\varepsilon<1/10$, it is possible to cover the whole plane with $\exp\exp\exp(\varepsilon^{-C})$ rotations of $E(\varepsilon)$.
\end{theorem}

We proceed by analyzing the same family of rotations that appears in Manners's proof, and the vast majority of our quantitative improvement takes place on the dynamics side (with a minor improvement in the deduction of the geometric result from the dynamical result).  
One key input is an adaptation of the entropic tools developed by Bourgain, Lindenstrauss, Michel, and Venkatesh~\cite{BLMV08} for quantitative $\times 2, \times 3$ problems.

The upper bound in Theorem~\ref{thm:main} is likely quite far from the truth.  The trivial volume-counting lower bound is $\varepsilon^{-1}/2$, and this seems difficult to improve.  Heuristics, kindly communicated to us by Freddie Manners, suggest that one might expect an upper bound on the order of $\varepsilon^{-1} \log (\varepsilon^{-1}) \log \log (\varepsilon^{-1})$.  Green~\cite{green2024100} proposes that an upper bound of $\varepsilon^{-C}$ would be a natural intermediate goal.


\subsection{Main dynamical result}

We follow the setup from Manners's paper~\cite[Section 5]{Man13}.  For $\theta$ a unit-norm complex number, let $$E_\theta(\varepsilon):=\theta^{-1}E(\varepsilon)=\{x \in \mathbb{C}: \Re(\theta z) \in (-\varepsilon, \varepsilon) \pmod{1}\}$$ denote the rotation of $E(\varepsilon)$ around the origin by an angle of $\arg(\theta^{-1})$.  As shown by Malikiosis, Matolcsi, and Ruzsa, we will need a combination of ``rational'' and ``irrational'' rotations.  For the rational rotations, define the Gaussian primes
$$P_5:=1+2i, \quad P_{13}:=2+3i$$
(so named because of their norms),
and define the unit-norm complex numbers
$$\theta_5: = P_5/\overline{P_5}, \quad \theta_{13} = P_{13}/\overline{P_{13}},$$
where as usual the bar denotes complex conjugation.  Our set of rational rotations is $$\Theta_N:= \{\theta_5^s\theta_{13}^t: 0 \leq s, t \leq N\}$$
for a suitable choice of $N$.  The rotations $\theta_5,\theta_{13}$ are ``rational'' in the sense of living in $\mathbb{Q}[i]$; in particular there is some (large) integer $M$ such that all of the pyjama stripes $E_\theta(\varepsilon)$ for $\theta \in \Theta_N$ are doubly-periodic with respect to the lattice $M\mathbb{Z}[i]$.

Although it is not the case that $\cup_{\theta \in \Theta_N} E_\theta(\varepsilon)=\mathbb{C}$, we will show that $\mathbb{C} \setminus \cup_{\theta \in \Theta_N} E_\theta(\varepsilon)$ is contained in a union of small balls around certain ``highly rational points'' (more or less the elements of $D\mathbb{Z}[i]$ for small natural numbers $D$).  Once such a ``rationality lemma'' is established, one can use an ``irrational trick'' (due originally to Malikiosis, Matolcsi, and Ruzsa) to cover the remaining points of $\mathbb{C}$ by adding a few irrational rotations.

For the rationality lemma, it is convenient to work in a dual formulation of the problem.  The continuous characters on $\mathbb{C}$ are the maps $\chi_w: \mathbb{Z} \to \mathbb{R}/\mathbb{Z}$ defined by
$$\chi_w(z):=\Re(wz) \pmod{1}$$
for $w \in \mathbb{C}$.  In this language, we can write
$$E_\theta(\varepsilon)=\{z \in \mathbb{C}: \chi_\theta(z) \in (-\varepsilon, \varepsilon) \}.$$
Dualizing, we see that the Pyjama Problem is asking for a finite set of $\theta$'s such that the union of the corresponding sets
$$E'_\theta(\varepsilon):=\{\chi_w: \chi_w(\theta) \in (-\varepsilon, \varepsilon)\}$$
contains all continuous characters.  The desired rationality lemma would say that $\cup_{\theta \in \Theta_N} E'_\theta(\varepsilon)$ contains $\chi_w$ for all $w \in \mathbb{C}$ outside of small balls around highly rational points.

We can now explain how dynamics enters the picture.  We have an action of $\mathbb{C}^\times$ on the space of continuous characters of $\mathbb{C}$ via
$$(\theta \cdot \chi_w)(z):=\chi_w(\theta z)$$
for $\theta \in \mathbb{C}^\times$ and $\chi_w$ a continuous character.  Now we can write
$$E'_\theta(\varepsilon)=\theta \cdot E_1'(\varepsilon),$$
so $\chi_w \in E'_\theta(\varepsilon)$ if and only if $\theta \cdot \chi_w \in E'_1(\varepsilon)$.  Thus $\cup_{\theta \in \Theta_N} E'_\theta(\varepsilon)$ consists of precisely the characters $\chi_w$ such that $\theta \cdot \chi_w \in E'_1(\varepsilon)$ for some $\theta \in \Theta_N$.  We think of starting with the ``point'' $\chi_w$ and then iteratively acting by the transformations $\theta_5, \theta_{13}$; the question of interest is which starting points have orbits that avoid the open set $E'_1(\varepsilon)$.  The analogy with the $\times 2, \times 3$ problem arises because $P_5, P_{13}$ are coprime.

These notions still make sense if one replaces the set of continuous characters on $\mathbb{C}$ with the space $\widehat{\mathbb{C}}$ of \emph{all} characters on $\mathbb{C}$ (additive homomorphisms from $\mathbb{C}$ to $\mathbb{R}/\mathbb{Z}$, not necessarily continuous).  Malikiosis, Matolcsi, and Ruzsa showed that the Pyjama Problem is equivalent to asking for a finite set of unit-norm complex numbers $\theta$ whose corresponding sets of characters
$\{\chi \in \widehat{\mathbb{C}}: \chi(\theta) \in (-\varepsilon, \varepsilon)\}$
together cover all of $\widehat{\mathbb{C}}$.  The advantage of this maneuver is that $\widehat{\mathbb{C}}$ is compact (since it is the Pontryagin dual of a discrete space).

As Manners describes in~\cite[Section 5.2]{Man13}, the space $\widehat{\mathbb{C}}$ is rather unwieldy because it is unmetrizable and generally ``too large'' to work with.  The solution is to study the dual space of a small subring $A$ of $\mathbb{C}$, rather than the dual space of all of $\mathbb{C}$.  Recall that in our dynamical perspective we are interested in orbits under the action of $\theta_5, \theta_{13}$.  For this action to be well-defined, the subring $A$ must be closed under multiplication by $\theta_5, \theta_{13}$.  With this in mind, we make the ``smallest possible choice''
$$A:=\mathbb{Z}[i][1/\overline{P_5}, 1/\overline{P_{13}}] \subseteq \mathbb{Q}(i)$$
(endowed with the discrete topology), and we define its Pontryagin dual
$$X:=\widehat{A}$$
(the space of all additive homomorphisms from $A$ to $\mathbb{R}/\mathbb{Z}$).  The space $X$ contains $\mathbb{C}$ as a subspace and is compact just like $\widehat{\mathbb{C}}$, but it is smaller and more ``manageable''.  For instance, we will see below that $X$ is metrizable and can be expressed as a quotient of $\mathbb{Q}_5 \times \mathbb{Q}_{13} \times \C$ with the fundamental domain $\mathbb{Z}_5 \times \mathbb{Z}_{13} \times [-1/2,1/2) \times [-1/2,1/2)i$.

In this setting, the pyjama stripe corresponds to the open set
$$E^*_1(\varepsilon):=\{\chi \in X: \chi(1) \in (-\varepsilon,\varepsilon) \pmod{1}\}.$$
It turns out that highly rational points of $\mathbb{C}$ correspond to low-order torsion points of $X$, so our task is to show that for each point $p \in X$, either $p$ is very close to a low-order torsion point or the orbit $\Theta_N \cdot p$ intersects the set $E^*_1(\varepsilon)$.  (Note that when we work with $\widehat{A}$ instead of $\widehat{\mathbb{C}}$, we cannot hope for finitely many rotations within elements of $A$ of $E^*_1(\varepsilon)$ to cover $\widehat{A}$, because if that were true, finitely many rational rotations of $E_1(\varepsilon)$ would be sufficient to cover $\mathbb{C}$.) This dichotomy is made precise in our main technical result, which is a quantitative analogue of~\cite[Lemma 5.2]{Man13}.

\begin{theorem}[Quantitative rationality lemma]\label{thm:quantitative-rationality}
There is a constant $C>0$ such that for any $0<\varepsilon, \delta<1/10$, the following holds with $$n=n_\varepsilon:=\exp\exp\exp(\varepsilon^{-C}) \quad \text{and} \quad N=N_{\varepsilon, \delta}:=\max(\exp\exp\exp(\varepsilon^{-C}), C\log(\delta^{-1})).$$ 
For each point $p \in X$, either $\Theta_N \cdot p$ intersects $E_1^*(\varepsilon)$, or there is a torsion point $r \in X$ of order at most $n$ such that $\dist(p, \overline{B}_{1/2}^{\mathbb{C}}(r)) \leq \delta$.
\end{theorem}

We emphasize that since $\Theta_N \subseteq \mathbb{Q}[i]$, the results of Malikiosis, Matolcsi, and Ruzsa guarantee that $\cup_{\theta \in \Theta_N} E^*_\theta(\varepsilon)$ cannot cover all of $X$.  Thus Theorem~\ref{thm:quantitative-rationality} (indeed, also Manners's non-quantitative version) is optimal in the qualitative sense of constraining the complement $X \setminus \cup_{\theta \in \Theta_N} E^*_\theta(\varepsilon)$ to lie in a very small set.  We remark that the independence of $n$ from $\delta$ is important in our application to Theorem~\ref{thm:main}.

\subsection{Organization}
In Section~\ref{sec:overview} we describe the proof strategy for Theorem~\ref{thm:quantitative-rationality}, including the analogy with the $\times 2, \times 3$ Problem.  In Section~\ref{s:notation} we establish useful properties of $A,X$ and prove other preliminary results.  We treat the three main stages of the proof (major arcs, minor arcs, and coda) in Sections~\ref{s:supermajor}, \ref{s:minor}, and \ref{s:maindeduction}, respectively.  In Section~\ref{s:irrationalargument} we transfer Theorem~\ref{thm:quantitative-rationality} back to the setting of $\mathbb{C}$ and apply the ``irrational trick'' to complete the proof of Theorem~\ref{thm:main}.  We make some final remarks in Section~\ref{sec:concluding}.

\section{General strategy}\label{sec:overview}

As in Manners's work, we exploit an analogy between the Pyjama Problem with the rotations $\Theta_N$ and Furstenberg's $\times 2, \times 3$ Problem.  In this section, we give some background for $\times 2, \times 3$-type problems; explain the central analogy; prove a quantitative strengthening of Furstenberg's $\times 2, \times 3$ Theorem (\'a la~\cite{BLMV08}); outline the analogous proof of Theorem~\ref{thm:quantitative-rationality}; and sketch the deduction of Theorem~\ref{thm:main} from Theorem~\ref{thm:quantitative-rationality}.  See Section~\ref{s:notation} for terminology of the space $X$ and its geometry.

\subsection{The $\times 2, \times 3$ Problem}
Let us briefly recall the setup of Furstenberg's $\times 2, \times 3$ Problem.  The doubling and tripling maps $T_2, T_3:\mathbb{R}/\mathbb{Z} \to \mathbb{R}/\mathbb{Z}$ are defined by $T_2(x):=2x$ and $T_3(x):=3x$.  A Borel subset $S \subseteq \mathbb{R}/\mathbb{Z}$ is \emph{$\times 2$-invariant} if $T_2^{-1}(A)=A$; $\times 3$-invariance is defined analogously.

There is a great variety of closed sets that are invariant under either $\times 2$ or $\times 3$ individually (for instance, various finite sets of rationals and Cantor sets).  Furstenberg's $\times 2, \times 3$ Theorem says that the range of closed sets invariant under \emph{both} $\times 2$ and $\times 3$ is much more restricted.  More precisely, this theorem states that if $S \subseteq \mathbb{R}/\mathbb{Z}$ is a $\times 2, \times 3$-invariant closed set, then either $S$ consists of a finite set of rationals or $S=\mathbb{R}/\mathbb{Z}$.  The theorem holds more generally with $2,3$ replaced by any other integers $a,b>1$ that are \emph{multiplicatively independent} in the the sense that $\log(a)/\log(b) \notin \mathbb{Q}$.

Furstenberg's celebrated $\times 2, \times 3$ Conjecture concerns not only the possible $\times 2, \times 3$-invariant sets but also the possible $\times 2, \times 3$-invariant measures on $\mathbb{R}/\mathbb{Z}$.  Recall that for a measure $\mu$ and a transformation $T$, we define $T\mu$ by $T\mu(S):=\mu(T^{-1}(S))$ for every measurable set $S$, and we say that $\mu$ is \emph{$T$-invariant} if $T\mu=\mu$.  A measure is \emph{atomic} if some point has strictly positive measure.  Furstenberg's $\times 2, \times 3$ Conjecture asserts that the Lebesgue measure is the only $T_2,T_3$-invariant nonatomic measure on $\mathbb{R}/\mathbb{Z}$.

Although the $\times 2, \times 3$ Conjecture remains open in general, Rudolph~\cite{rudolph19902} has proven it under the additional assumption that $\mu$ has positive entropy (see also~\cite{johnson1992measures,parry1996squaring,host1995nombres}).  More recently, Bourgain, Lindenstrauss, Michel, and Venkatesh \cite{BLMV08} have proven a ``quantitative'' version of Rudolph's result.  One consequence is the following quantitative strengthening of Furstenberg's $\times 2, \times 3$ Theorem.  Write $\Psi_N=:\{2^s 3^t: 0 \leq s,t \leq N\}$.

\begin{theorem}\label{thm:quantitative-furstenberg}
There is a constant $C>0$ such that for any $\varepsilon,\delta>0$, the following holds with
$$n=n_\varepsilon:=\exp\exp\exp(\varepsilon^{-C}) \quad \text{and} \quad N=N_{\varepsilon,\delta}:=\exp\exp(\varepsilon^{-C})+\log_2(\delta^{-1}).$$
For each point $x \in \R/\Z$, either $\Psi_N \cdot x$ is $\varepsilon$-dense in $\R/\Z$, or or $x$ lies within distance $\delta$ of some rational with denominator at most $n$.
\end{theorem}


\subsection{The analogy}

The ``dictionary'' between the $\times 2, \times 3$ Problem and the Pyjama Problem with rotations $\Theta_N$ goes as follows:
\[
\begin{aligned}
\widehat{\mathbb{Z}}=\mathbb{R}/\mathbb{Z} \;&\longleftrightarrow\; X=\widehat{A}; \\
\times 2, \times 3
\;&\longleftrightarrow\;
\times \theta_5, \times \theta_{13}; \\
\text{rational point of $\mathbb{R}/\mathbb{Z}$} \;&\longleftrightarrow\; \text{complex ball around torsion point of $X$}; \\
\text{denominator of a rational point} \;&\longleftrightarrow\; \text{order of a torsion point};\\
\text{orbit is $\varepsilon$-dense} \;&\longleftrightarrow\; \text{orbit intersects $E_1^*(\varepsilon)$}.
\end{aligned}
\]
One should think of $\mathbb{R}/\mathbb{Z},A$ as metric measure spaces and of $(\mathbb{R}/\mathbb{Z},\times 2, \times 3),(X, \times \theta_5, \times \theta_{13})$ as topological dynamical systems.  The Gaussian rationals $\theta_5, \theta_{13} \in A$ are ``coprime'' in a sense analogous to the sense in which the integers $2,3$ are coprime.

The points in $\mathbb{R}/\mathbb{Z}$ with non-dense $\times 2, \times 3$-orbits are precisely the rationals (in fact these orbits are finite).  The situation is more complicated in $X$, which looks locally like a piece of $\mathbb{Q}_5 \times \mathbb{Q}_{13} \times \C$ (with $\times \theta_5, \times \theta_{13}$ acting by component-wise multiplication).  Here, the points with finite orbits are the torsion points of $X$, but a complex ball around a torsion point also has its orbit contained in a finite union of disks, essentially because $\times \theta_5, \times \theta_{13}$ are ``non-expansive'' in the $\mathbb{C}$-coordinate.

In their respective settings, Theorems~\ref{thm:quantitative-furstenberg} and~\ref{thm:quantitative-rationality} say that the \emph{finite} orbit of a point becomes fairly dense unless the point lies very close to a substructure whose \emph{infinite} orbit is contained in a fairly small compact set.  The Pyjama notion of ``fairly dense'' is weaker than density in the full space $X$; it is merely the condition that we need for proving Theorem~\ref{thm:main}.

To finish putting our new work in context, let us state Manners's rationality lemma, which is infinitary just like Furstenberg's $\times 2, \times 3$ Theorem.
\begin{theorem}[Manners's rationality lemma]
If $S \subseteq X$ is a closed $\times \theta_5,\times \theta_{13}$-invariant set, then either $S = X$, or $S \subseteq \cup_{j = 1}^L B_R^\mathbb{C}(q_j)$ for some $R > 0$ and some torsion points $q_1, \dots, q_L \in X$.
\end{theorem}
In summary, Theorem~\ref{thm:quantitative-rationality} bears the same relation to Manners's rationality lemma as Theorem~\ref{thm:quantitative-furstenberg} bears to Furstenberg's $\times 2, \times 3$ Theorem.

\subsection{Quantitative $\times 2, \times 3$}

We now describe the deduction of Theorem~\ref{thm:quantitative-furstenberg} from the main results of \cite{BLMV08}; we will employ a similar strategy for Theorem~\ref{thm:quantitative-rationality}.  We require two inputs from \cite{BLMV08}.  
The first input is the following corollary of the entropy-based result~\cite[Theorem 1.4]{BLMV08} (see the deduction of Corollary 1.6 from Theorem 1.4 in that paper).  For $M \in \mathbb{N}$, let $\mathcal{P}_M$ denote the partition of $\mathbb{R}/\mathbb{Z}$ into the $M$ equal-length intervals $P_i:=[i/M, (i+1)/M)$.

\begin{theorem}\label{thm:BLMV-entropy}
There is a constant $c>0$ such that the following holds for every $M \in \mathbb{N}$ coprime to $6$, and for every $0<\rho<1$.  If $S \subseteq \mathbb{R}/\mathbb{Z}$ is a set that intersects at least $M^\rho$ parts of the partition $\mathcal{P}_M$, then the set
$$\{2^s 3^t x: s,t \in \mathbb{Z}_{\geq 0},~ 2^s 3^t< M,~ x \in S\}$$
is $(\log M)^{-c\rho}$-dense in $\mathbb{R}/\mathbb{Z}$.
\end{theorem}
The second input is the following consequence of Baker's Theorem, as stated in~\cite[Corollary 4.4]{BLMV08}.

\begin{theorem}\label{thm:BLMV-baker}
There is a constant $c>0$ such that the following holds.  Let $a_1<a_2<\cdots$ denote the elements of $\{2^s 3^t: s,t \in \mathbb{Z}_{\geq 0}\}$ in increasing order.  Then
$$a_{k+1}-a_k \leq \frac{a_k}{c (\log a_k)^c}$$
for all integers $k \geq 2$.
\end{theorem}
We are ready for the proof of Theorem~\ref{thm:quantitative-furstenberg}.  We take suitable $N_0,N_1<N$ and work with the shorter orbits $\Psi_{N_0} \cdot x, \Psi_{N_0+N_1} \cdot x$ before studying the full orbit $\Psi_N \cdot x$.  Analyzing the shortest orbit $\Psi_{N_0} \cdot x$ lets us distinguish ``major-arc'' and ``minor-arc'' cases.  The longer orbit $\Psi_{N_0+N_1} \cdot x$ figures in an important intermediary conclusion in the minor-arc case before we can harness the power of the full orbit $\Psi_N \cdot x$.

\begin{proof}[Proof of Theorem~\ref{thm:quantitative-furstenberg}]
Let $c$ denote the minimum of the two constants $c$ from Theorems~\ref{thm:BLMV-entropy} and~\ref{thm:BLMV-baker}.  Set $N_0:=\exp\exp(\varepsilon^{-C_1})$, for a constant $C_1>0$ to be determined later, and consider the orbit $\Psi_{N_0} \cdot x$.  The pigeonhole principle provides some $0 \leq s_1,s_2,t_1,t_2 \leq N_0$, with $(s_1,t_1) \neq (s_2,t_2)$, such that $2^{s_1} 3^{t_1} x,2^{s_2}3^{t_2}x$ lie at most $N_0^{-2}$ apart.  Thus $r:=2^{s_1} 3^{t_1} -2^{s_2}3^{t_2}$ is a nonzero integer with $|r| \leq 6^{N_0}$ such that
$$\|rx\|_{\mathbb{R}/\mathbb{Z}} \leq N_0^{-2}.$$
We distinguish two cases according to whether or not $\|rx\|$ is extremely small.

{\bf Major arcs.} Suppose that $\|rx\| \leq \delta$.  Dividing through by $r$, we find that $x$ lies within $\delta$ of some rational of height at most $6^{N_0}$, as desired.

{\bf Minor arcs.} Suppose that $\|rx\| \geq \delta$.  Theorem~\ref{thm:BLMV-baker} tells us that the set $\Psi_{\log_2(\|rx\|^{-1})}$ is $\frac{\|rx\|^{-1}}{c(\log \|rx\|^{-1})^c}$-dense in the interval $[0,\|rx\|^{-1}]$.  It follows that the orbit $\Psi_{N_1} \cdot rx$ is $\eta$-dense in $\R/\Z$, where we have set $$N_1:=\log_2(\|rx\|^{-1}) \quad \text{and} \quad \eta:=\frac{1}{c(\log \|rx\|^{-1})^c}.$$  Our bounds $\delta \leq \|rx\| \leq N_0^{-2}$ give $\log_2(\|rx\|^{-1}) \leq \log_2(\delta^{-1})$ and $\eta \leq \frac{1}{c(2 \log N_0)^c}$.


Next, set $M:=\lfloor \eta^{-1} \rfloor$.  The key observation is that, by the above considerations, the difference set
$$(\Psi_{N_0+N_1} \cdot x)-(\Psi_{N_0+N_1} \cdot x) \supseteq \Psi_{N_1} \cdot rx$$
intersects every part of the partition $\mathcal{P}_M$.  Let
$$Q:=\{m \in [1,M]: (\Psi_{N_0+N_1} \cdot x) \cap P_m \neq \emptyset\}$$
index the parts of $\mathcal{P}_M$ occupied by elements of $\Psi_{N_0+N_1} \cdot x$.  Since $P_m-P_{m'}$ intersects only two parts of the partition $\mathcal{P}_M$ for each choice of $m,m'$, we deduce that $|Q| \gg M^{1/2}$, say, $|Q| \geq M^{0.49}$.  We can now apply Theorem~\ref{thm:BLMV-entropy} to conclude that $\Psi_{N_0+N_1+\log_2(M)} \cdot x$
is $(\log M)^{-0.49c}$-dense in $\mathbb{R}/\mathbb{Z}$.  We conclude on noting that $(\log M)^{-0.49c}\leq \varepsilon$ as long as $C_1$ is sufficiently large relative to $c$, and that
$$N_0+N_1+\log_2(M) \leq \exp\exp(\varepsilon^{-C})+\log_2(\delta^{-1}),$$
for $C$ sufficiently large relative to $C_1$.
\end{proof}

This argument requires $N$ to be at least double-exponential in $\varepsilon^{-C}$.  One exponential is due to the entropy argument in the proof of Theorem~\ref{thm:BLMV-entropy}, and the other is due to the logarithm in the Baker-type result Theorem~\ref{thm:BLMV-baker}.  

\subsection{Proof outline for Theorem~\ref{thm:quantitative-rationality}}  Our proof of Theorem~\ref{thm:quantitative-rationality} has the same overall shape as the proof of Theorem~\ref{thm:quantitative-furstenberg}, but the details differ throughout because the space $X$ is more complicated than $\mathbb{R}/\mathbb{Z}$.  Here we will sketch the main steps without carefully tracking quantitative dependences.

Set $N_0:=\exp\exp(\varepsilon^{-C_1})$, for a suitable constant $C_1>0$, and consider the orbit $\Theta_{N_0} \cdot p$.  Pigeonholing produces some nonzero $r=\theta_5^{s_1}\theta_{13}^{t_1}-
\theta_5^{s_2}\theta_{13}^{t_2}$ with $0 \leq s_1,s_2,t_1,t_2$ such that
$$\dist(rp,0) \ll N_0^{-1/2}.$$
We condition on whether $rp$ is extremely close to $B^{\C}_1(0)$ or merely fairly close.

{\bf Major arcs.}  Suppose that $rp$ in fact lies extremely close to $B^{\C}_1(0)$ (say, within distance $\exp(-\exp(200C_2 N_0))$, for another constant $C_2>0$).  Since $r$ is a Gaussian rational of height at most $\exp(100N_0)$, clearing denominators produces a torsion point $q \in X$ of order at most $\exp(200N_0)$ such that $p$ lies extremely close to $B^{\C}_{\exp(200N_0)}(q)$.  In contrast to the $\times 2, \times 3$ setting, this conclusion is too weak because we want a complex ball of radius $1/2$ rather than of radius $\exp(200N_0)$.  Suppose that $p$ lies extremely close to the complex circle of radius $R$ around $q$.  A new element of our work is showing that if $R\geq 1/2$, then geometric considerations (see Lemma~\ref{lem:sphereintersection}) force $\Theta_N \cdot p$ to intersect $E_1^*(\varepsilon)$.\footnote{We emphasize that this orbit will not be $\varepsilon$-dense in $X$ if $R$ is of only moderate size.  If we wanted to take $\varepsilon$-density of $\Theta_N \cdot p$ in $X$ as the first possible outcome of Theorem~\ref{thm:quantitative-rationality}, then in the second possible outcome we would need to allow the the radius of the complex ball to grow with $\varepsilon^{-1}$; we will not pursue this direction further.}

{\bf Minor arcs.} Suppose instead that $rp$ is not extremely close to $B^{\C}_1(0)$ (say, distance at least $\exp(-\exp(200C_2 N_0))$).  Using a quantitative version of Manners's analysis of so-called ``non-archimedean limit points'' of $X$ (see Theorem~\ref{thm:mannersapproximation}), we deduce that
$$(\Theta_{N_0+N_1} \cdot p)-(\Theta_{N_0+N_1} \cdot p)$$
is $\eta$-dense in $X$, for $\eta:=\exp(-\varepsilon^{-C_2})$ and $N_1:=2\exp(200C_2 N_0)$.\footnote{Our Theorem~\ref{thm:mannersapproximation} plays the role that Baker's Theorem (via Theorem~\ref{thm:BLMV-baker}) played in the analogous part of the proof of Theorem~\ref{thm:quantitative-furstenberg}.  Baker's Theorem does still appear, however, in our analysis of the major-arc case and in our deduction of Theorem~\ref{thm:main} from Theorem~\ref{thm:quantitative-rationality}.}

We now arrive at a delicate point.  In the $\times 2, \times 3$ setting, there was an obvious way to partition $\mathbb{R}/\mathbb{Z}$ into equal-length intervals.  The space $X$ is not so symmetric, however, and the $\times \theta_5, \times \theta_{13}$ actions have different ``stretching'' effects on the $\mathbb{C}, \mathbb{Q}_5, \mathbb{Q}_{13}$-components of $X$.  To make the subsequent entropy machinery work, we must construct a highly ``skewed'' partition $\mathcal{P}$ (see Proposition~\ref{prop:constructionofpartition} and Theorem~\ref{thm:entropyweirdpartition}) such that each part contains an $\eta$-ball and has rather particular width in each of the $\mathbb{C}, \mathbb{Q}_5, \mathbb{Q}_{13}$-directions.

As in the proof of Theorem~\ref{thm:quantitative-furstenberg}, we deduce that $\Theta_{N_0+N_1} \cdot p$ occupies $\gg |\mathcal{P}|^{1/2}$ parts of the partition $\mathcal{P}$.  This puts us in a position to apply our analogue of Theorem~\ref{thm:BLMV-entropy} (see Corollary~\ref{cor:quantitativedenseness}), whose proof requires adapting the machinery of~\cite{BLMV08} from $\mathbb{R}/\mathbb{Z}$ to $X$; this step is fairly involved.  The upshot is that $\Theta_{N_0+2N_1} \cdot p$ is $\varepsilon$-dense in $X$, and in particular intersects $E^*_1(\varepsilon)$.  This completes our sketch proof of Theorem~\ref{thm:quantitative-rationality} with $N:=N_0+2N_1$. 

\subsection{Back to pyjama stripes}
To conclude this overview, we sketch the deduction of Theorem~\ref{thm:main} from Theorem~\ref{thm:quantitative-rationality}.  This part of the argument follows~\cite{Man13} quite closely.  Theorem~\ref{thm:quantitative-rationality}, applied with some $\delta$ much smaller than $\varepsilon$, tells us that $\cup_{\theta \in \Theta_N} E^*_{\theta}(\varepsilon)$ covers all of $X=\widehat{A}$ with the possible exception of some complex balls around torsion points of small order.  Unraveling definitions, we find that $\cup_{\theta \in \Theta_N} E_{\theta}(\varepsilon)$ covers all of $\mathbb{C}$ with the possible exception of some small balls around the elements of the (near-)lattices $\frac{D}{m}\mathbb{Z}[i] \setminus \mathbb{Z}[i]$, for a large Gaussian integer $D$ and small natural numbers $m$ (see Lemma~\ref{lem:lattice-points}).  These small gaps reflect the arithmetic obstructions that arise when one uses only rational rotations (recall that $\theta_5,\theta_{13} \in \mathbb{Q}(i)$).  To fill in the remaining gaps, we replace $\Theta_N$ by the union of several carefully-chosen irrational rotations of $\Theta_N$.

\section{The ring $A$ and the space $X$}\label{s:notation}

Before we dive into the proof proper, we establish several algebraic properties of the ring $A$ and several geometric properties of the space $X$.  Section~\ref{sec:X-setup} contains a careful geometric description of $X$, and Section~\ref{sec:fundamental-domains} constructs fundamental domains of certain quotients of $X$ which play a central role in our later entropy arguments.  We suggest skimming the other two subsections (on the structure of finite quotients of $A$ and some Fourier analysis on $X$) on a first reading and referring back to them as necessary.


\subsection{The group structure of $(\theta_5^{-n}A)/A$}

Recall that $A:=\mathbb{Z}[i][1/\overline{P_5}, 1/\overline{P_{13}}]$.  We first show that $(\theta_5^{-n}A)/A$ is isomorphic to $\mathbb{Z}/5^n\mathbb{Z}$.

\begin{lemma}\label{lem:equivalence}
Let $n$ be a nonnegative integer.
\begin{enumerate}
    \item We have $P_5^{-n}\Z[i]+A=\theta_5^{-n}A+A$, i.e., $(P_5^{-n}\mathbb{Z}[i])/A = (\theta_5^{-n}A)/A$.
    \item We have $(P_5^{-n}\mathbb{Z}[i])/A \cong \mathbb{Z}[i]/P_5^{n}\mathbb{Z}[i] \cong \mathbb{Z}/5^n\mathbb{Z}$.  Moreover, the latter isomorphism can be chosen to intertwine the $\times P_5^\ell$ and $\times 5^\ell$ maps for all integers $0 \leq \ell \leq n$.
\end{enumerate}
\end{lemma}

\begin{proof}
We start with part (i).  From $$P_5^{-n}\Z[i]=\theta_5^{-n} \overline{P_5}^n \Z[i] \subseteq \theta_5^{-n}\Z[i] \subseteq \theta_5^{-n}A$$ we see immediately that $P_5^{-n}\Z[i]+A \subseteq \theta_5^{-n}A+A$.  It remains to show the reverse inclusion, namely, that $\theta_5^{-n}a \in P_5^{-n}\Z[i]+A$ for any $a \in A$.  Write $$a=\overline{P_5}^{k} \overline{P_{13}}^m b,$$ where $b \in \Z[i]$ is coprime to $\overline{P_5}, \overline{P_{13}}$ and $k,m \in \Z$.  We wish to find $c \in A$ such that $\theta_5^{-n}a+c \in P_5^{-n}\Z[i]$, i.e.
$$\overline{P_5}^{k+n} \overline{P_{13}}^m b+P_5^n c \in \Z[i].$$
Take $c$ to be of the form $c=\overline{P_5}^{k+n} \overline{P_{13}}^m d$, for $d \in \Z[i]$; we wish to choose $d$ such that
$$\overline{P_5}^{k+n} \overline{P_{13}}^m(b+P_5^n d) \in \Z[i].$$
Since $P_5$ is coprime to $\overline{P_5}, \overline{P_{13}}$, we can choose $d \in \Z[i]$ so that $b+P_5^n d$ is divisible by an arbitrarily large power of $\overline{P_5}\overline{P_{13}}$.  In particular, we can ensure that $\overline{P_5}^{k+n} \overline{P_{13}}^m(b+P_5^n d) \in \Z[i]$, as desired.

We now turn to part (ii).  The first isomorphism is given by $$P_5^{-n}x+A \mapsto x+P_5^n \Z[i]$$ for $x \in \Z[i]$.  The second isomorphism is classical and follows from the observation that
$$\Z[i], 1+\Z[i], 2+\Z[i], \ldots, (5^n-1)+\Z[i]$$
is a complete set of representatives for $\Z[i]/P_5^n \Z[i]$; the desired isomorphism maps $$m+\Z[i] \mapsto m+5^n \Z$$ for $0 \leq m \leq 5^n-1$.  The intertwining of the multiplication maps is obvious.
\end{proof}

Let $S_{13}(n)$ denote the multiplicative subgroup of $A/\theta_5^n A$ generated by $\theta_{13}$.  The following lemma is a manifestation of the ``lifting the exponent'' trick.

\begin{lemma}\label{lem:boundedsubgroup}
There is a nonnegative integer $\alpha$ (independent of $n$) such that
$$S_{13}(n) \supseteq 1 + \theta_5^{\alpha} (A/\theta_5^nA) = 1 + P_5^\alpha (A/P_5^n A)$$
for all nonnegative integers $n$.
\end{lemma}
\begin{proof}
The exponential and logarithm maps give an isomorphism between the groups $(1 + P_5(A/P_5^nA), \times)$ and $(P_5(A/P_5^nA), +)$. It follows from part (ii) of Lemma~\ref{lem:equivalence} that every additive subgroup of $(P_5(A/P_5^nA), +)$ is of the form $(P_5^\ell(A/P_5^n A), +)$ for some $\ell \ge 1$; the corresponding multiplicative subgroup of $(1 + P_5(A/P_5^nA), \times)$ is $(1 + P_5^\ell(A/P_5^nA), \times)$.

Let $k$ denote the multiplicative order of $P_{13}$ modulo $P_5 \Z[i]$, and let $S'_{13}(n)$ denote the multiplicative subgroup of $S_{13}(n)$ generated by $\theta_{13}^k$.  The choice of $k$ guarantees that $S'_{13}(n)$ is contained in $(1 + P_5(A/P_5^nA), \times)$ and hence is equal to $(1 + P_5^\ell(A/P_5^nA), \times)$ for some $\ell$ (possibly depending on $n$).  Now set $\alpha$ to be $v_{P_5}(\theta_{13}^k - 1)$ (the largest $m$ such that $P_5^m$ divides $\theta_{13}^k - 1$), so that $\theta_{13}^k \notin 1+ P_5^{\alpha+1}(A/P_5^n A)$.  It follows that $\ell \leq \alpha$.
\end{proof}

The following lemma is the analogue for the ring $A$ of the natural isomorphism between the dual group of $(N^{-1}\Z)/\Z$ and the group $\Z/NZ$.

\begin{lemma}\label{lem:dual}
The dual group of $(\theta_5^{-n}A)/A$ is isomorphic to $A/\theta_5^nA$ via the map $\psi$ sending the element $\eta \in A/\theta_5^nA$ to the character $$x \mapsto \{\iota_{\mathbb{Q}_5}(\eta \cdot x)\}_5 + \{\iota_{\mathbb{Q}_{13}}(\eta \cdot x)\}_{13} - \Re(\eta x) \pmod{1}$$
(for $x \in A\theta_5^{-n}/A$).
\end{lemma}

\begin{proof}
We first check that the map $\psi$ is well-defined.  Let $\eta \in A$ and $x \in \theta_5^{-n}A$.  For $a \in \Q(i)$, it was shown in \cite[Proposition 6.2]{Man13} that $$\{\iota_{\mathbb{Q}_5}(a)\}_5 + \{\iota_{\mathbb{Q}_{13}}(a)\}_{13} - \Re(a) \in \mathbb{Z}$$ if and only if $a \in A$.  It follows that the expression
$$\{\iota_{\mathbb{Q}_5}(\eta \cdot x)\}_5 + \{\iota_{\mathbb{Q}_{13}}(\eta \cdot x)\}_{13} - \Re(\eta x) \pmod{1}$$
remains unchanged if one shifts $\eta$ by an element of $\theta_5^n A$ (since the product of such a shift with $x$ lies in $A$) or if one shifts $x$ by an element of $A$ (since again the product of such a shift with $\eta$ lies in $A$).  So $\psi$ is well-defined.  Also, $\psi$ is additive since $\iota_{\mathbb{Q}_5}, \iota_{\mathbb{Q}_{13}}$, $\{\cdot\}_5$, $\{\cdot\}_{13}, \Re$ are all additive modulo $1$.

It remains to check that $\psi$ is an isomorphism.  Since $|(\theta_5^{-n}A)/A|=|A/\theta_5^nA|=5^n$ by Lemma~\ref{lem:equivalence}, and finite abelian groups and their duals have equal cardinalities, it suffices to show that $\psi$ is injective.  Suppose $\psi(\eta)=0$.  Then evaluating $\psi(\eta)$ at $x=\theta_5^{-n}$ gives
$$\{\iota_{\mathbb{Q}_5}(\eta \theta_5^{-n})\}_5 + \{\iota_{\mathbb{Q}_{13}}(\eta \theta_5^{-n})\}_{13} - \Re(\eta \theta_5^{-n}) \in \Z,$$
and another application of~\cite[Proposition 6.2]{Man13} implies that $\eta\theta_5^{-n} \in A$, i.e., $\eta \in \theta_5^n A$.
\end{proof}

\subsection{The space $X$}\label{sec:X-setup}
Recall that $X = \widehat{A}$ is the Pontryagin dual of $A$.  Our task is characterizing $X$ more explicitly.

There are two square roots of $-1$ in $\Z_5$, one equivalent to $2 \pmod{5}$ and one equivalent to $-2 \pmod{5}$; let $\iota_{5}: \mathbb{Q}(i) \to \mathbb{Q}_{5}$ be the inclusion sending $i$ to the square root of $-1$ that is $-2 \pmod{5}$.  Notice that $|\iota_5(P_5)|_5=1$ and $|\iota_5(\overline{P_5})|_5=5^{-1}$ (where $|\cdot|_5$ is the usual $5$-adic absolute value). In the same way, fix the embedding $\iota_{13}: \mathbb{Q}(i) \to \mathbb{Q}_{13}$ satisfying $|\iota_{13}(P_{13})|_{13}=1$ and $|\iota_{13}(\overline{P_{13}})|_{13}={13}^{-1}$.

The ``integer part'' of an element $z=\sum_{k=m}^\infty a_k 5^k \in \Q_5$ is
$$\lfloor z \rfloor_5:=\sum_{k=0}^\infty a_k 5^k \in \mathbb{Z}_5,$$
and the ``fractional part'' $\{z\}_5:=z-\lfloor z \rfloor_5$ naturally lives in $[0,1)$.  We have a well-defined map $j_5:\Q_5 \to X$ where $j_5(z)$ is the homomorphism
$$x \mapsto \{z\iota_5(x)\}_5$$
for $x \in A$ (and we have identified $[0,1)$ with $\R/\Z$ in the natural way).  The map $j_{13}:\Q_{13} \to X$ is defined analogously.  We also have a map $j_{\C}:\C \to X$ where $j_{\C}(z)$ is the homomorphism
$$x \mapsto \Re(zx) \pmod{1}.$$
Combining these three maps $j_5,j_{13},j_{\C}$, we can define the map $j: \Q_5 \times \Q_{13} \times \C \to X$ via
$$j(a,b,z):=j_5(a)+j_{13}(b)-j_{\C}(z);$$
notice the negative sign before $j_{\C}(z)$.  We also have the \emph{diagonal embedding} $\iota^{\Delta}: \Q(i) \to \Q_5 \times \Q_{13} \times \C$ via
$$\iota^{\Delta}(z):=(\iota_5(z),\iota_{13}(z),z).$$
The following important result from \cite[Proposition 6.2]{Man13} says that $j$ surjects onto $X$ with kernel $\iota^{\Delta}(A)$; we will say more about fundamental domains shortly.

\begin{lemma}\label{lem:freddie-isom}
The map $j$ induces an isomorphism $(\mathbb{Q}_5 \times \mathbb{Q}_{13} \times \mathbb{C})/
\iota^{\Delta}(A) \to X$ with fundamental domain $\mathbb{Z}_5 \times \mathbb{Z}_{13} \times [a, a + 1) \times [b, b + 1)i$ (for any $a, b \in \mathbb{R}$).
\end{lemma}

We can turn $X$ into a metric space by defining the distance
$$\dist(x, y):= \inf \{|a|_5 + |b|_{13} + |z|_\mathbb{C}: j(a, b, z) = x - y\}$$
for $x,y \in X$.  Within a small ball around $0$, distances can be computed in the fundamental domain $\mathbb{Z}_5 \times \mathbb{Z}_{13} \times [1/2, 1/2) \times [-1/2,1/2)i$.

For $\rho>0$ and $x \in X$, write $B^{\C}_\rho(x):=x+j_{\C}(B_\rho(0))$ and $S^{\C}_\rho(x):=x+j_{\C}(S_\rho(0))$ for the complex ball and complex circle of radius $\rho$ around $x$.

Finally, we define the \emph{order} of a torsion point $q \in X$, denoted $\mathrm{ord}(q)$, to be the least positive integer $n$ such that $nq = 0$.  The \emph{$n$-torsion points} of $X$ are the torsion points of order exactly $n$.  It follows from Lemma~\ref{lem:freddie-isom} that the torsion points of $X$ are precisely the points $q=j(\iota^{\Delta}(z))$ for $z \in \Q(i)$, and that $\ord(q)$ is the smallest positive integer $n$ such that $nz \in A$.





\subsection{Fundamental domains}\label{sec:fundamental-domains}
The following proposition constructs fundamental domains of quotients of $X$ with specified diameters in the $\Q_5,\Q_{13},\C$-components.  The $n=0$ case amounts to the second part of Lemma~\ref{lem:freddie-isom} above.  For $n,\alpha, \beta \in \Z_{\geq 0}$, define the subset
$$F(n,\alpha,\beta):=5^\alpha \mathbb{Z}_5\times 13^\beta \mathbb{Z}_{13} \times P_5^{-n}\overline{P_5}^\alpha P_{13}^\beta ([0,1) \times [0,1)i) \subseteq \Q_5 \times \Q_{13} \times \C.$$

\begin{proposition}\label{prop:constructionofpartition}
Let $n \in \mathbb{Z}_{\geq 0}$, and set $N:=5^n$.  Suppose $N_5, N_{13}, N_{\mathbb{C}}>0$ satisfy $N_5N_{13}N_{\mathbb{C}}^2 \geq 130N^{-1}$.  Set $\alpha:=-\lfloor \log_5 N_5 \rfloor$ and $\beta:=-\lfloor \log_{13} N_{13} \rfloor$.  Then $F=F(n,\alpha,\beta)$ is a fundamental domain of $X/\theta_5^{-n}A$ satisfying the following.
\begin{enumerate}[(i)]
    \item $\mathrm{diam}_{\mathbb{Q}_5}(F) \le N_5$;
    \item $\mathrm{diam}_{\mathbb{Q}_{13}}(F) \le N_{13}$;
    \item $\mathrm{diam}_{\mathbb{C}}(F) \le N_{\mathbb{C}}$.
\end{enumerate}
Let $\mathcal{P}$ denote the partition of $X$ given by the $\theta_5^nA$-orbits of $F$.  Then for each choice of $P,P' \in \mathcal{P}$, the difference set $P-P'$ intersects only $4$ elements of $\mathcal{P}$.
\end{proposition}

\begin{proof}
By definition, we have $5^{-\alpha} \leq N_5$ and $13^{-\beta} \leq N_{13}$, and we have $5^{-\alpha}13^{-\beta}N_{\mathbb{C}}^2 \geq 2N^{-1}$.  It is clear that $F$ has the desired diameter in the $\mathbb{Q}_5,\mathbb{Q}_{13}$-components.  And the diameter in the $\mathbb{C}$-component is
$$|\sqrt{2} \cdot P_5^{-n}\overline{P_5}^\alpha P_{13}^\beta|=\sqrt{2 \cdot N^{-1} \cdot 5^\alpha 13^\beta} \leq N_\mathbb{C}.$$
Recall that $X=(\mathbb{Q}_5 \times \mathbb{Q}_{13} \times \mathbb{C})/A$, so $$X/\theta_5^{-n}A=(\mathbb{Q}_5 \times \mathbb{Q}_{13} \times \mathbb{C})/\theta_5^{-n}A=(\mathbb{Q}_5 \times \mathbb{Q}_{13} \times \mathbb{C})/P_5^{-n}A.$$
Since $A$ has index $5^n=N$ as a subgroup of $\theta_5^{-n}A$ and the latter acts freely on $\mathbb{Q}_5 \times \mathbb{Q}_{13} \times \mathbb{C}$, we see that $\vol(X/\theta_5^{-n}A)=N^{-1}\vol(X)=N^{-1}$.  The volume of $F$ is also $N^{-1}$, so to conclude that $F$ is a fundamental domain for $X/\theta_5^{-n}A$ it suffices to show that every element of $\mathbb{Q}_5 \times \mathbb{Q}_{13} \times \mathbb{C}$ has a representative modulo $\theta_5^{-n}A$ lying in $F$.

Let $(a,b,z) \in \mathbb{Q}_5 \times \mathbb{Q}_{13} \times \mathbb{C}$.  By subtracting some element of $A$, we can find a representative $(a',b',z')$ of $(a,b,z)$ such that $a' \in 5^{\alpha}\mathbb{Z}_5$ and $b' \in 13^{\beta} \mathbb{Z}_{13}$.  Now notice that
$$P_5^{-n}A \cap 5^\alpha \mathbb{Z}_5 \cap 13^\beta \mathbb{Z}_{13}=P_5^{-n}\overline{P_5}^\alpha P_{13}^\beta \mathbb{Z}[i].$$
It follows that by further subtracting an element of $P_5^{-n}\overline{P_5}^\alpha P_{13}^\beta \mathbb{Z}[i]$, we can find a representative $(a'',b'',z'')$ of $(a,b,z)$ such that $a'' \in 5^{\alpha}\mathbb{Z}_5$ and $b'' \in 13^{\beta} \mathbb{Z}_{13}$, and $z''$ lies in any fixed fundamental domain of $P_5^{-n}\overline{P_5}^\alpha P_{13}^\beta \mathbb{Z}[i]$ (viewed as a lattice in $\mathbb{C}$); one such fundamental domain is $P_5^{-n}\overline{P_5}^\alpha P_{13}^\beta ([0,1) \times [0,1)i)$, so we have succeeded in finding a representative of $(a,b,z)$ in $F$.  This completes the proof of the first part of the proposition.

The partition $\mathcal{P}$ consists of the fundamental domains
$$F_w:=5^\alpha \mathbb{Z}_5 \times 13^\beta \mathbb{Z}_{13} \times P_5^{-n}\overline{P_5}^\alpha P_{13}^\beta (w+[0,1) \times [0,1)i)$$
for $w \in \mathbb{Z}[i]$.  Then
\begin{align*}
F_w-F_{w'} &=5^\alpha \mathbb{Z}_5 \times 13^\beta \mathbb{Z}_{13} \times  P_5^{-n}\overline{P_5}^\alpha P_{13}^\beta (w-w'+[-1,1) \times [-1,1)i)\\
 &=F_{w-w'} \cup F_{w-w'-1} \cup F_{w-w'-i} \cup F_{w-w'-1-i},
\end{align*}
which proves the second part of the proposition.
\end{proof}
We remark that the fundamental domain $F$ from the proposition has in-radius $\asymp \min(N_5,N_{13},N_{\C})$.



\subsection{Fourier approximation, smoothness, and bump functions}
In this subsection we collect a few analytic notions that we will need to use later.  We start with a standard lemma about Fourier approximation of functions on $\R/\Z$.

\begin{lemma}\label{lem:quantitativefourier}
Let $0<\eta<1/10$ and $M>0$, and set $L:=\lceil M\eta^{-1} \rceil$.  Then for every smooth function $f: \mathbb{R}/\mathbb{Z} \to \mathbb{R}$ with $C^2$-norm at most $M$, there are coefficients $a_{-L},a_{-L+1},\ldots, a_L \in \mathbb{R}$ such that the function $$g(x):=\sum_{\xi=-L}^L a_\xi e(\xi x)$$ satisfies $\|f-g\|_{L^\infty(\R/\Z)} \leq \eta$.  Moreover, we may ensure that $a_0=\int_{\R/\Z}f(x) \,dx$ and that $|a_\xi| \leq M$ for all nonzero $\xi$.
\end{lemma}

\begin{proof}
Consider the Fourier coefficients
$$a_\xi:=\int_{\R/\Z} f(x)e(-\xi x) \,d\xi$$
for $\xi \in \mathbb{Z}$.  Notice that $a_0=\int_{\R/\Z} f(x) \,dx$.  For $\xi \neq 0$, integration by parts and the pointwise inequality $|f''(x)| \leq \|f\|_{C^2(\R/\Z)} \leq M$ let us bound
$$|a_\xi|=\left|\int_{\R/\Z} f(x)e(-\xi x) \,dx \right|= \left|\int_{\R/\Z} f''(x)e(-\xi x)(-2\pi i \xi)^{-2} \,dx \right| \leq \frac{M}{4\pi^2 |\xi|^2}.$$
Now Fourier inversion gives
$$f(x)=g(x)+\sum_{|\xi|>L} a_\xi e(\xi x),$$
where $g(x):=\sum_{\xi=-L}^L a_\xi e(\xi x)$.  It remains to check that $\|f-g\|_{L^\infty(\R/\Z)}$ is small.  Indeed, for each $x \in \R/\Z$, we have
$$|f(x)-g(x)|=\left| \sum_{|\xi|>L} a_\xi e(\xi x) \right| \leq \sum_{|\xi|>L} \frac{M}{4\pi^2 |\xi|^2} \leq \frac{M}{2\pi^2 L} \leq \eta$$
by our choice of $L$.
\end{proof}

We will now discuss the space $C^\infty(X)$ of smooth functions on $X$. Smooth functions are defined on both $\mathbb{C}$ and $\mathbb{Q}_p$.  In the former case a smooth function is one which is infinitely $\R$-differentiable, and in the latter case a smooth function is one that is locally constant.  As $X$ locally looks like $\mathbb{Z}_5 \times \mathbb{Z}_{13} \times [-1/2,1/2) \times [-1/2,1/2)i$, a smooth function on $X$ is a function $f$ that is locally smooth on each component.  Let $\Lip(f):=\sup_{x \neq y}\frac{|f(x)-f(y)|}{\dist(x,y)}$ denote the \emph{Lipschitz constant} of $f$.

We will later work with the \emph{Sobolev norm} $\| \cdot \|_{H^{10}}$ on functions $f:X \to \mathbb{R}$ given by $$\|f\|_{H^{10}}^2:=\sum_{\eta \in A} \langle \eta \rangle^{10} |\widehat{f}(\eta)|^2,$$ where we have written $\langle \eta \rangle:=|\eta|_5+|\eta|_{13}+|\eta|_\mathbb{C}$ for brevity.  (The constant $10$ is not important.)  We will need the fact that $\|f\|_{H^{10}}$ is not too big when $f$ is a smooth bump function at scale $\gamma$.

\begin{lemma}\label{lem:smoothapproximation}
Fix a smooth function $g: \mathbb{R} \to \mathbb{R}_{\geq 0}$ satisfying $g(x)=1$ for $|x| \leq 1/10$ and $g(x)=0$ for $|x| \geq 1/5$.  Then there is a constant $C>0$ (depending only on $g$) such that the following holds.  Let $0<\gamma<1/100$ and $u,v \in \mathbb{N}$.  Define the function $f: X \to \mathbb{R}$ by
$$f(a,b,z,w):={\bf 1}_{a \equiv 0 \pmod{5^u}} \cdot {\bf 1}_{b \equiv 0 \pmod{13^v}} \cdot g(z/\gamma)g(w/\gamma)$$
for $(a,b,z,w)$ in the fundamental domain $\mathbb{Z}_5 \times \mathbb{Z}_{13} \times [-1/2,1/2) \times [-1/2,1/2)i$.  Then
$$\|f\|_{H^{10}}^2 \leq C \gamma^{-28}uv \cdot 5^{7u} \cdot 13^{7v} \quad \text{and} \quad \Lip(f) \leq C(\gamma^{-1}+5^u+13^v).$$
\end{lemma}

\begin{proof}
We begin with the Sobolev norm bound.  We can write each nonzero $\eta \in A$ uniquely as $\eta=\overline{P_5}^\alpha \overline{P_{13}}^\beta x$, where $\alpha,\beta \in \mathbb{Z}_{\leq 0}$ and $x \in \mathbb{Z}[i]$ are such that $x$ is coprime to $\overline{P_5}$ if $\alpha<0$ and $x$ is coprime to $\overline{P_{13}}$ if $\beta<0$.
Then $$\langle \eta \rangle\asymp 5^{-\alpha}+13^{-\beta}+5^{\alpha/2} \cdot 13^{\beta/2} \cdot |x|.$$
Evaluating $\widehat{f}(\eta)$ on the fundamental domain $\mathbb{Z}_5 \times \mathbb{Z}_{13} \times [0,1) \times [0,1)i$, we can expand
\begin{align*}
\widehat{f}(\eta) &=\int_{\mathbb{Z}_5} e(-\{a\eta\}_5) {\bf 1}_{a \equiv 0 \pmod{5^u}} \,da  \cdot \int_{b \in \mathbb{Z}_{13}} e(-\{b\eta\}_{13}){\bf 1}_{b \equiv 0 \pmod{13^v}} \,db\\
 & \qquad \cdot \int_{[0,1)} e(\Re(\eta )w)g(z/\gamma) \,dw \cdot \int_{[0,1)} e(-\Im(\eta)w)g(w/\gamma) \,dz.
\end{align*}
The first integral is trivially bounded by $1$, and it vanishes when $\alpha<-u$; likewise for the second integral.  By repeated integration by parts, the third integral can be bounded as $\ll_M \gamma^{-M}(1+|\Re(\eta)|)^{-M}$ for any $M \in \mathbb{N}$; likewise for the fourth integral.  Thus
\begin{align*}
\widehat{f}(\eta) &\ll_M {\bf 1}_{\alpha \geq -u} \cdot {\bf 1}_{\beta \geq -v} \cdot \gamma^{-2M} (1+|\Re(\eta)|)^{-M}(1+|\Im(\eta)|)^{-M}\\
 &\ll_M {\bf 1}_{\alpha \geq -u} \cdot {\bf 1}_{\beta \geq -v} \cdot \gamma^{-2M} (1+|\eta|)|)^{-M}\\
 &\leq {\bf 1}_{\alpha \geq -u} \cdot {\bf 1}_{\beta \geq -v} \cdot \gamma^{-2M} \cdot 5^{-\alpha M/2} \cdot 13^{-\beta M/2} (1+|x|)|)^{-M}.
\end{align*}
So we can bound $\|f\|_{H^{10}}^2$ as
\begin{align*}
\|f\|_{H^{10}}^2 &\ll_M \gamma^{-4M} \sum_{\alpha=-u}^0 \sum_{\beta=-v}^0 \sum_{x \in \mathbb{Z}[i]} 5^{-\alpha M}13^{-\beta M} (1+|x|)^{-2M} (5^{-\alpha}+13^{-\beta}+5^{\alpha/2} \cdot 13^{\beta/2} \cdot |x|)^{10}\\
 &\leq \gamma^{-4M} uv \cdot 5^{uM} \cdot 13^{vM} \sum_{x \in \mathbb{Z}[i]} (1+|x|)^{10-2M}.
\end{align*}
The inner sum is $\ll 1$ if we choose $M=7$ (say), and we conclude that
$$\|f\|_{H^{10}}^2 \ll \gamma^{-28}uv \cdot 5^{7u} \cdot 13^{7v},$$
as desired.

The Lipschitz constant bound follows from the fact that within a small ball around $0$, distances can be computed in the fundamental domain $\mathbb{Z}_5 \times \mathbb{Z}_{13} \times [-1/2,1/2) \times [-1/2,1/2)i$.
\end{proof}

\section{Major 
arcs}\label{s:supermajor}
The major-arc case is where $\Theta_{N_0}p-\Theta_{N_0}p$ contains a point $(\theta_5^s \theta_{13}^t-\theta_5^{s'}\theta_{13}^{t'})p$ that not only has size $\ll N_0^{-2}$ but also lies extremely close to $B^{\C}_{1}(0)$.  We would like to ``divide by $\theta_5^s \theta_{13}^t-\theta_5^{s'}\theta_{13}^{t'}$'' in order to deduce that $p$ lies near a small complex ball around a torsion point of small order.  This operation is trickier here than in the $\times 2, \times 3$ setting because $\theta_5^s \theta_{13}^t-\theta_5^{s'}\theta_{13}^{t'}$ is a Gaussian rational rather than an integer.  The naive approach would tell us that $p$ lies near a very large complex ball around a torsion point of small order.  In order to control the size of the complex ball, we will need some additional arguments.


\subsection{Baker-type results}\label{sec:baker}
The first result we require is Baker's Theorem \cite{Baker1, Baker2, Baker3}, which provides a quantitative sense in which (base-$e$) logarithms of algebraic numbers are badly approximable by rationals.  Let $\mathbb{L}:=\{\lambda \in \mathbb{C}: e^\lambda \in \overline{\mathbb{Q}}\}$.  Write $H(\beta)$ for the height of an algebraic number $\beta$.

\begin{theorem}[Baker's Theorem]\label{thm:baker}
Suppose that $\lambda_1, \ldots, \lambda_n \in \mathbb{L}$ are linearly independent over the rationals.  Then there is a constant $c=c(\lambda_1, \ldots, \lambda_n)>0$ such that for any algebraic numbers $\beta_1, \ldots, \beta_n$ not all equal to $0$, we have
$$|\beta_1 \lambda_1+\cdots+\beta_n \lambda_n| \geq c\max(H(\beta_1), \ldots, H(\beta_n))^{-c}.$$
\end{theorem}
The constant $c(\lambda_1, \ldots, \lambda_n)$ is effectively computable.

We will need the following consequence of Baker's Theorem.  Take $\alpha_5,\alpha_{13} \in [-1/2,1/2)$ such that $e^{2\pi i \alpha_5}=\theta_5$ and $e^{2\pi i \alpha_{13}}=\theta_{13}$.

\begin{corollary}\label{cor:baker}
There is a constant $c>0$ such that for every nonzero rational $r$ of height at most $H$, we have
$$\|\alpha_5 - r\|_{\mathbb{R}/\mathbb{Z}}, \|\alpha_{13} - r\|_{\mathbb{R}/\mathbb{Z}} \ge cH^{-c}.$$
\end{corollary}

\begin{proof}
We will prove the result for $\alpha_5$; the argument for $\alpha_{13}$ is identical.  Notice that $$e^{2\pi i \alpha_5}=\theta_5 \in \mathbb{Q}[i] \subseteq \overline{\mathbb{Q}} \quad \text{and} \quad e^{2\pi i}=1 \in \overline{\mathbb{Q}},$$ so that $2\pi i \alpha_5, 2\pi i \in \mathbb{L}$.  We claim that $2\pi i \alpha_5, 2\pi i$ are linearly independent over the rationals.  Indeed, if they were linearly dependent, then there would be some natural number $n$ with $2\pi i n \alpha_5$ an integer multiple of $2\pi i$.  Then we would have $\theta_5^n=1$, i.e., $P_5^n=\overline{P_5}^n$, but this is impossible by unique prime factorization in $\mathbb{Z}[i]$ (recall that $P_5, \overline{P_5}$ are coprime).

Thus we can apply Theorem~\ref{thm:baker}, which provides $c>0$ such that
$$|2\pi i q \alpha_5 -2\pi i p| \geq c\max(|p|,|q|)^{-c}$$
for all integers $p,q$ with $q \neq 0$.  In particular, dividing through by $2\pi i q$, we find that
$$|\alpha_5-p/q| \geq c\max(|p|,|q|)^{-c}.$$
Now let $r$ be a nonzero rational of height at most $H$ (so $H>0$).  Then there is some integer shift $r'$ of $r$ such that $\|\alpha_5-r\|_{\mathbb{R}/\mathbb{Z}}=|\alpha_5-r'|$ and the height of $r'$ is at most $2H$.  Now, writing $r'=p/q$, we conclude that
$$\|\alpha_5-r\|_{\mathbb{R}/\mathbb{Z}}=|\alpha_5-r'| \geq c(2H)^{-c},$$
as desired (up to adjusting the value of $c$).
\end{proof}


\begin{lemma}[Quantitative approximation lemma]\label{lem:quantitativeapproximation}
There is a constant $c > 0$ such that the following holds.  If $K,\ell \in \mathbb{N}$ satisfy $0 < \ell \le K^c$, then the sets $\{n\ell \alpha_5: 0 \leq n \leq K\}$ and $\{n\ell \alpha_{13}: 0 \leq n \leq K\}$ are both $K^{-c}$-dense in $\mathbb{R}/\mathbb{Z}$.
\end{lemma}

\begin{proof}
We will prove the lemma only for $\alpha_5$, since again the argument for $\alpha_{13}$ is identical.  Set $\varepsilon:=K^{-c}$.  We will show that the set $\{n\ell \alpha_5: 0 \leq n \leq K\}$ intersects every interval $I \subseteq \mathbb{R}/\mathbb{Z}$ of length $\varepsilon$.  To this end, let $\rho: \mathbb{R}/\mathbb{Z} \to \mathbb{R}_{\geq 0}$ be a smooth function supported on $I$ that is equal to $1$ on the length-$\varepsilon/2$ interval in the center of $I$ and has smoothness norm $\|\rho\|_{C^2(\R/\Z)} \leq 10 \varepsilon^{-1/2}$.  It suffices to show that $\E_{0 \leq n \leq K} \rho(n\ell \alpha_5)>0$.  By Lemma~\ref{lem:quantitativefourier} with $M:=10\varepsilon^{-1/2}$ and $\eta:=\varepsilon/4$, we can write
$$\rho(x) = \sum_{|\xi| \leq 50\varepsilon^{-3/2}} a_\xi e(\xi x) + h(x),$$
where $a_0 = \int_{\mathbb{R}/\mathbb{Z}} \rho(x) \,dx \geq \varepsilon/2$, and every $|a_\xi| \le 10 \varepsilon^{-1/2}$, and $\|h\|_{L^\infty} \le \varepsilon/4$.  It follows that
$$\E_{0 \leq n \leq K} \rho(n\ell \alpha_5)=\E_{0 \leq n \leq K} \left[a_0+\sum_{0<|\xi|\leq 50\varepsilon^{-3/2}}a_\xi e(n\ell \alpha_5 \xi)+h(n\ell \alpha_5)\right].$$
The $a_0$ contribution is at least $\varepsilon/2$, and the $h$ contribution is most $\varepsilon/4$ in absolute value, so we will be done once we show that the contribution from the $a_\xi$'s ($\xi \neq 0$) is smaller than $\varepsilon/4$.  In particular, we will show that the contribution from each individual $a_\xi$ is smaller than $\varepsilon^{5/2}/400$.

Fix some nonzero $\xi$ with $|\xi| \leq 50\varepsilon^{-3/2}$.  By Corollary~\ref{cor:baker}, the contribution of $a_\xi$ can be bounded as
$$\left|\E_{0 \leq n \leq K} a_\xi e(n\ell \alpha_5 \xi)\right| \ll \frac{|a_\xi|}{(K+1)\|\ell \alpha_5 \xi\|_{\mathbb{R}/\mathbb{Z}}} \leq \frac{|a_\xi|}{(K+1) c'(\ell |\xi|)^{-c'}},$$
where $c'$ is the constant from Corollary~\ref{cor:baker}.  Substituting the bounds $|a_\xi| \leq 10 \varepsilon^{-1}$ and $\ell \leq K^c$ and $|\xi| \ll \varepsilon^{-3/2}$, we find that
$$\left|\E_{0 \leq n \leq K} a_\xi e(n\ell \alpha_5 \xi)\right| \ll \frac{\varepsilon^{-1}(K^c \varepsilon^{-3/2})^{c'})}{K}=\varepsilon^{-1-3c'/2}K^{-1+cc'}=\varepsilon^{-1-3c'/2+1/c-c'}.$$
As long as $c$ is sufficiently small relative to $c'$, this last exponent will be larger than $5/2$, with enough room to overwhelm the implied constant.  This completes the proof.
\end{proof}

\subsection{Geometric lemmas}

We isolate three geometric lemmas that will be useful later in the proof.  The first lemma shows that if a small (Gaussian rational) multiple of a point $x \in X$ is small, then $x$ is close to a torsion point of small order.  By the \emph{height} of a Gaussian rational we mean the maximum of the absolute values of its numerator and denominator.

\begin{lemma}\label{lem:rationalclosetoball}
Let $\eta,k>0$ be parameters. 
If the point $x \in X$ satisfies $\dist(rx, \overline B_1^{\mathbb{C}}(0))\leq \eta$ for some nonzero Gaussian rational $r \in \mathbb{Q}(i)$ of height at most $k$, then there is a torsion point $q \in X$ of order at most $k^2$ such that
$$\dist(x, \overline B_{k}^\mathbb{C}(q)) \le (2k^2 + k)\eta.$$
\end{lemma}
\begin{proof}
Express $rx = j(a, b, z)$ for some $(a,b,z) \in \Q_5 \times \Q_{13} \times \C$ with $|a|_5, |b|_{13} \leq \eta $ and $|z| \leq 1 + \eta$.  
Notice that $|\iota_5(r)|_5 \geq k^{-2}$ and hence $a':=a/\iota_5(r)$ satisfies $|a'|_5=|a|_5 \cdot |\iota_5(r)|_5^{-1} \leq \eta k^2$.
Likewise $b':=b/\iota_{13}(r) \in \mathbb{Z}_{13}$ satisfies $|b'|_{13} \leq \eta k^2$.  Set $z':=z/r$, so that $|z'| \leq k|z| \leq k+\eta k$.  Now
$$r(x-j(a',b',z'))=j(a-ra',b-rb',z'rz')=0$$
shows that $q:=x-j(a',b',z')$ is a torsion point of order at most $k^2$ (since $(r\overline{r})q=0$).  Finally, $x-q=j(a',b',z')$ has $\Q_5,\Q_{13}$-parts of size at most $\eta k^2$ and has $\mathbb{C}$-part of size at most $k+\eta k$, so $x$ lies within distance $(2k^2+k)\eta$ of $\overline{B}_k^{\C}(q)$, as desired.
\end{proof}

The next lemma tells us that any torsion point is invariant under some power of $\theta_5$ once we multiply through to ``clear denominators''.
\begin{lemma}\label{lem:torsion-clear-denom}
Let $k \in \mathbb{N}$.  If $q \in X$ is a torsion point of order $k$, then there are integers $0 \leq s_0 <k^2$ and $1\leq s_1 \leq k^2$ such that $q':=\theta_5^{s_0}q$ satisfies $\theta_5^{s_1}q'=q'$.  The analogous statement with $5$ replaced by $13$ also holds.
\end{lemma}

\begin{proof}
Recall (e.g., \cite[Proposition 6.2]{Man13}) that the map $j \circ \iota^{\Delta}$ provides a $\Q(i)$-module isomorphism between $\Q(i)/A$ and of torsion points of $X$.  In particular, $j \circ \iota^{\Delta}$ lets us identify $(\frac{1}{k}A)/A$ with the group $G$ of torsion points of $X$ of order dividing $k$.  Notice that $|G|=|(\frac{1}{k}A)/A|\leq k^2$ (more precisely, $|G|=k^2 \cdot \|k\|_5 \cdot \|k\|_{13}$), and that multiplication by $\theta_5$ maps $G$ into itself.  Thus, by the pigeonhole principle, there are  $0 \leq s_0<s \leq k^2$ such that $\theta_5^{s_0}q=\theta_5^{s}q$; we conclude the lemma on setting $s_1:=s-s_0$. 
\end{proof}

The last lemma says that every sufficiently large complex circle in $X$ intersects the width-$0$ pyjama stripe $E_1^*(0)$.

\begin{lemma}\label{lem:sphereintersection}
For every radius $\rho\geq 1/2$ and every point $x \in X$, the complex circle $S_\rho^{\mathbb{C}}(x)$ intersects the line $E^*_1(0)=\{\chi \in X: \chi(1)=0\}$.
\end{lemma}
\begin{proof}
Recall that $S_\rho^{\C}(x)=\{x+j_{\C}(z): j \in \C, |z|=\rho\}$. We can evaluate the character $x+j_{\C}(z)$ at the element $1 \in A$ as
$$(x+j_{\C}(z))(1)=x(1)+\Re(z) \pmod{1}.$$
Since $\rho \geq 1/2$, we can find some $z$ with $\Re(z)=-x(1) \pmod{1}$ and $|z|=\rho$.  This choice of $z$ gives $(x+j_{\C}(z))(1)=0 \pmod{1}$ and thus exhibits a point in $S_\rho^{\C}(x) \cap E_1^*(0)$.
\end{proof}

\subsection{The main argument}
We are now ready to assemble the pieces.  We start with the hypothesis that $(\theta_5^{s_1} \theta_{13}^{t_1}-\theta_5^{s_2}\theta_{13}^{t_2})p$ is very close to $B_{1/10}^{\C}(0)$.  We deduce from Lemma~\ref{lem:rationalclosetoball} that $p$ is very close to a complex ball, possibly with large radius, around a torsion point of small order.  If this radius is larger than $1/2$, then we can use an approximate version of Lemma~\ref{lem:sphereintersection} to conclude that some $\theta_5^{s_3}\theta_{13}^{t_3}p$ is close to $E_{1/2}^*(0)$. The details are as follows.

\begin{proposition}[Major arcs]\label{prop:super-major}
There is a constant $C>0$ such that the following holds for all $0<\delta,\varepsilon<1/10$ and $k \in \mathbb{N}$.  Set $K:=(k\varepsilon^{-1})^C$ and $\eta:=\min(\delta/(3k^2), \exp(-10K))$.  Suppose $x \in X$ satisfies $\dist(rx, \overline{B}_1^{\C}(0)) \leq \eta$ for some nonzero Gaussian rational $r$ of height at most $k$.  Then either $\dist(x,\overline{B}_{1/2}^{\C}(q)) \leq \delta$ for some torsion point $q \in X$ of order at most $k^2$, or $\Theta_{K}x$ intersects $E^*_1(\varepsilon)$.
\end{proposition}

\begin{proof}
Lemma~\ref{lem:rationalclosetoball} produces a torsion point $q \in X$ of order at most $k^2$ such that
$$\dist(x,\overline{B}_k^{\C}(q)) \leq (2k^2+k)\eta \leq 3k^2 \eta.$$
Let us write
$$x=q+w+\sigma,$$
where $w \in S^{\C}_R(0)$ for some $0 \leq R \leq k$, and $\dist(\sigma,0) \leq 3k^2 \eta$.  If $R \leq 1/2$, then we obtain the first conclusion of the lemma (recall that $3k^2\eta \leq \delta$ by the choice of $\eta$).  For the remainder of the proof, we will consider the case where $R \geq 1/2$; we will show that $\Theta_K x \cap E^*_1(\varepsilon) \neq \emptyset$.

Lemma~\ref{lem:torsion-clear-denom} provides $0 \leq s_0<k^4$ and $1 \leq s_1 \leq k^4$ such that $q':=\theta_5^{s_0}q$ satisfies $\theta_5^{s_1}q'=q'$.  Now consider the set
$$Z:=\{\theta_5^{s_0+s_1u}x: 0 \leq u \leq K/(2s_1)\} \subseteq \Theta_{K}x.$$
It suffices to show that $Z$ intersects $E^*_1(\varepsilon)$.  Write
$$\theta_5^{s_0+s_1u}x=\theta_5^{s_0+s_1u}(q+w+\sigma)=q'+\theta_5^{s_0+s_1u}w+\theta_5^{s_0+s_1u}\sigma.$$
For all $0 \leq u \leq K/(2s_1)$, we have $$\dist(\theta_5^{s_0+s_1u}\sigma,0) \leq 5^{s_0+s_1u} \dist(\sigma,0) \leq 5^{s_0+s_1u} \cdot 3 k^2 \eta \leq \exp(-K)<\varepsilon/10$$
(with room to spare); this will turn out to be a negligible error term.

Next, Lemma~\ref{lem:quantitativeapproximation} tells us that $$\{\theta_5^{s_0+s_1 u}: 0\leq u \leq K/(2s_1)\}$$
is $(K/(2s_1))^{-c}$-dense in $\R/\Z$, where $c>0$ is the constant from that lemma.  It follows that
$$\{\theta_5^{s_0+s_1 u}w: 0\leq u \leq K/(2s_1)\}$$
is $(K/(2s_1))^{-c}R$-dense in $S^{\C}_R(0)$.  Notice that $(K/(2s_1))^{-c}R<\varepsilon/10$ as long as $C$ is sufficiently large relative to $1/c$.  By the Triangle Inequality, we know that the $\varepsilon/5$-neighborhood of $Z$ contains $q'+S_R^{\C}(0)=S_R^{\C}(q')$.  Lemma~\ref{lem:sphereintersection} tells us that the latter set intersects $E^*_1(0)$, so we conclude that $Z$ intersects $E^*_1(\varepsilon)$, as desired.
\end{proof}

\section{Minor arcs}\label{s:minor}

The minor-arc case is where $\Theta_{N_0}p-\Theta_{N_0}p$ contains a point $(\theta_5^s \theta_{13}^t-\theta_5^{s'}\theta_{13}^{t'})p$ of size $\ll N_0^{-2}$ whose distance to $B^{\C}_{1}(0)$ is not too tiny.  As described above, our approach is very similar to our approach for the analogous step in the $\times 2, \times 3$ setting.  One important technical detail is that when we apply the entropy machinery of~\cite{BLMV08}, we will have to use a partition that is very ``skewed''.

\subsection{Manners's denseness criterion}

One of the key steps in Manners's proof is the analysis of what he calls ``non-Archimedean limit points'' in $X$.  
The objective of this subsection is to prove the following quantitative version of one of Manners's results in this direction (see~\cite[Lemma 7.1]{Man13}).

\begin{theorem}[Manners's denseness criterion]\label{thm:mannersapproximation}
There is a constant $C > 0$ such that the following holds.  Let $0<\eta, \varpi <1/10$.  If $x \in X$ can be written as $x = j(a, b, z)$ with
$$\varpi \le |a|_{5} \le \eta/10, \quad |b|_{13} \le 13^{-C \eta^{-5}}, \quad \text{and} \quad |z| \le \eta/10,$$
then the orbit
$$\{\theta_5^s \theta_{13}^t x: 0 \leq s \leq \log_5(\varpi^{-1} \eta^{-4})+C, ~ 0 \leq t \leq C\eta^{-5}\}$$
is $\eta$-dense in $X$.  The analogous result with the roles of $5,13$ swapped also holds.
\end{theorem}

We will see in the proof that in fact only a single value of $s$ (of size $\log_5(\varpi^{-1} \eta^{-2})+O(1)$) is necessary.  Unfortunately, the argument is somewhat technical and involves keeping track of the relative sizes of many parameters.

\begin{proof}
Before we begin the main argument, we apply the ``lifting the exponent'' trick, in the style of Lemma~\ref{lem:boundedsubgroup}.  Let $n \in \N$ be a parameter to be chosen later, and let $\pi_n: \mathbb{Z}_5 \to \mathbb{Z}/5^n \mathbb{Z}$ denote the projection map.  Note that $\iota_5(\theta_{13}) \in \mathbb{Z}_5 \setminus 5\mathbb{Z}_5$; we will henceforth write just $\theta_{13}$ instead of $\iota_5(\theta_{13})$ for typographical clarity.  Arguing as in the second paragraph of the proof of Lemma~\ref{lem:boundedsubgroup}, we obtain a non-negative integer $\alpha$ (independent of $n$) such that
$$\pi_n(\{\theta_{13}^t: 0 \leq t \leq 5^n-1\}) \supseteq 1+5^\alpha \mathbb{Z}/5^n \Z.$$
In particular, for each $w \in 1+5^{\alpha} \mathbb{Z}_5$, there is some $0 \leq t \leq 5^n-1$ such that $w-\theta_{13}^t \in 5^n \mathbb{Z}_5$.

We extend this observation as follows.  Let $\ell \in \mathbb{Z}$ be such that $$a \in 5^\ell \Z_5 \setminus 5^{\ell+1}\Z_5.$$  The hypothesis on $|a|_5$ implies that $\ell \leq \log_5(\varpi^{-1})$.
Let $m \in \mathbb{N}$ be a parameter to be chosen later.  Recall that $\theta_5 \in 5^{-1}\Z_5 \setminus \Z_5$, so $\theta_5^m \in 5^{-m}\Z_5 \setminus 5^{-m+1}\Z_5$.  Thus the projection of
$$\{\theta_5^m \theta_{13}^t a: 0 \leq t \leq 5^n-1\}$$
onto $5^{\ell-m}\Z/ 5^{n+\ell-m}\Z$ contains an entire coset $\widetilde{a}+5^{\ell-m+\alpha}\Z$ of $5^{\ell-m+\alpha}\Z / 5^{n+\ell-m}\Z$, and for each $w \in \widetilde{a}+5^{\ell-m+\alpha}\Z_5$, there is some $0 \leq t \leq 5^n-1$ such that $w-\theta_5^m \theta_{13}^t a \in 5^{n+\ell-m} \mathbb{Z}_5$.

We now turn to the main argument.  Fix any point $x' \in X$.  We will find $t$ such that $\dist(x',\theta_5^m \theta_{13}^t x) \leq \eta$.  Applying Proposition~\ref{prop:constructionofpartition}, we can write $x'=j(a',b',z')$ with
$$|a'|_5 \leq C_1 \eta^{-3}, \quad |b'|_{13} \leq \eta/10, \quad \text{and} \quad |z'| \leq \eta/10,$$
where $C_1>0$ is some absolute constant (e.g., $C_1=130000$ suffices).  
Notice that
$$a' \in 5^{\ell'}\mathbb{Z}_5$$
for $\ell':=\lceil -\log_5(C_1 \eta^{-3})\rceil=\log_5(\eta^3)+O(1)$.  Set
$$m:=\ell+\log_5(\eta^{-4})+C_2,$$
for a constant $C_2 \geq 0$ to be chosen later.  Notice that $\ell-m \leq \ell'$ and so $a' \in 5^{\ell-m}\Z_5$.

The key step is shifting the representation $(a',b',z')$ of $x'$ by $(q,q,q)$, for some suitable element $q \in A$.  We will choose $q$ so that $|q|_{13}, |q|$ are very small and $a'-q$ lies in $\widetilde{a}+5^{\ell-m+\alpha}\Z_5$.  Then we will be able to choose $t$ so that $(a'-q)-\theta_5^m \theta_{13}^t a \in 5^{n+\ell-m} \mathbb{Z}_5$, and we will conclude that $\dist(x',\theta_5^m \theta_{13}^t x) \leq \eta$.  The details are as follows.

Recall that $\overline{P_5} \in 5^{-1}\Z_5 \setminus \Z_5$ and $13 \in \Z_5 \setminus 5\Z_5$.  Set
$$\beta:=\lceil \log_{13}(5^{\alpha+1}\eta^{-1}) \rceil.$$
Now take the unique integer $0 \leq a^* \leq 5^\alpha-1$ such that
$$q:=\overline{P_5}^{\ell-m} \cdot 13^\beta a^* \in A$$
satisfies
$$a'-q \in \widetilde{a}+5^{\ell-m+\alpha}\Z_5.$$
Next, using the argument from the second paragraph of the proof, take $0 \leq t \leq 5^n-1$ such that
\begin{equation}\label{eq:choice-of-t}
(a'-q)-\theta_5^m \theta_{13}^t a \in 5^{n+\ell-m} \mathbb{Z}_5.
\end{equation}
We are ready to assemble the pieces.  Since $x'=j(a'-q,b'-q,z'-q)$, we have
$$\dist(x',\theta_5^m \theta_{13}^t x) \leq |(a'-q)-\theta_5^m \theta_{13}^t a|_5+|(b'-q)-\theta_5^m \theta_{13}^t b|_{13}+|(z'-q)-\theta_5^m \theta_{13}^t z|.$$
We will bound each contribution individually.

Set the parameter
$$n:=\log_5 (\eta^{-5})+C_2+1,$$
where $C_3$ is a constant (depending on $C_2$) to be determined later.  For the $\Q_5$-contribution, \eqref{eq:choice-of-t} gives
$$|(a'-q)-\theta_5^m \theta_{13}^t a|_5 \leq 5^{-n-\ell+m}=\eta/5$$
which is at most $\eta/5$ (say) as For the $\Q_{13}$-contribution, the Triangle Inequality gives
$$|(b'-q)-\theta_5^m \theta_{13}^t b|_{13} \leq \max(\eta/10, 5^\alpha \cdot 13^{-\beta}, 13^t \cdot 13^{-C \eta^{-4}}) 
.$$
The second term is at most $\eta/10$ by the choice of $\beta$.  The third is at most $13^{5^n-C\eta^{-5}}$, and this is comfortably smaller than $\eta/10$ as long as $C$ is large enough relative to $C_3$.
Finally, for the $\C$-contribution, the Triangle Inequality gives
$$|(z'-q)-\theta_5^m \theta_{13}^t z|_{\C} \leq \eta/10+5^{(\ell-m)/2} \cdot 13^\beta \cdot 5^\alpha+\eta/10 \leq \eta/5+O(5^{-C_2}\eta).$$
The last of these three terms is at most $\eta/5$ (giving a total of at most $2\eta/5$ for the $\C$-contribution) as long as $C_2$ is sufficiently large.  Combining all three of these contributions, we conclude that
$$\dist(x',\theta_5^m \theta_{13}^t x) \leq \eta/5+\eta/5+2\eta/5<\eta.$$
This concludes the proof.
\end{proof}

\subsection{Entropy input}\label{subs:entropy}
The objective of this subsection is to prove Theorem~\ref{thm:entropyweirdpartition}, which is the ``$\times \theta_5, \times \theta_{13}$ analogue'' of the main $\times 2, \times 3$ result of~\cite[Theorem 3.2]{BLMV08}.  We will follow the structure of the argument from~\cite{BLMV08} quite closely.  Before we can state our theorem, we need several pieces of notation.

First, for $\mathcal{P}$ a finite partition of $X$ and $\mu$ a probability measure on $X$, let $$H_{\mu}(\mathcal{P}):=-\sum_{P \in \mathcal{P}}\mu(P) \log(\mu(P))$$ (with the convention $0 \log(0)=0$) denote the \emph{entropy} of $\mathcal{P}$ with respect to $\mu$.

Second, for $\mu$ a measure on $X$ and $s,t$ nonnegative integers, we write $(\theta_5^s \theta_{13}^t)_*\mu$ for the pushforward of $\mu$ under the $\times \theta_5^s \theta_{13}^t$ map; explicitly, if we denote the $\times \theta_5^s \theta_{13}^t$ map by $\psi$, then $((\theta_5^s \theta_{13}^t)_*\mu)(B):=\mu(\psi^{-1}(B))$ for every measurable set $B \subseteq X$.

Finally, for $f: X \to \C$ a smooth function, write $\mu(f):=\int_X f(x) \,d\mu(x)$.  In this notation, we have $((\theta_5^s \theta_{13}^t)_*\mu)(f)=\int_X f(\theta_5^s \theta_{13}^t x) \, d\mu(x)$.  We use $m$ to denote the probability Haar measure on $X$ (defined on the fundamental domain $\Z_5 \times \Z_{13} \times [0,1) \times [0,1)i$ as the product of the Haar measures on the individual components). Recall the Sobolev norm $\|f\|_{H^{10}}^2:=\sum_{\eta \in A} \langle \eta \rangle^{10} |\widehat{f}(\eta)|^2$ and the shorthand $\langle \eta \rangle:=|\eta|_5+|\eta|_{13}+|\eta|_\mathbb{C}$.

\begin{theorem}\label{thm:entropyweirdpartition}
There is a constant $\kappa>0$ such that the following holds.
Suppose $\rho, \delta>0$ and $n,T \in \mathbb{N}$ satisfy
$$\frac{10}{\log_5(T)} \leq\delta \leq \frac{\rho}{10} \quad \text{and} \quad 5^{20/\delta} \leq T \leq \frac{\log 5}{4\log 13}\cdot \delta n.$$
Set $N:=5^n$, and let $\mathcal{P}_{N,T}$ denote the partition produced by Proposition~\ref{prop:constructionofpartition} with the parameters $$(N_5,N_{13},N_{\C}):=\left(\frac{13^{T/4}}{N^{1+3\delta/4}}, \frac{1}{N^{\delta/4} \cdot 13^{3T/4}}, \sqrt{130} \cdot N^{\delta/2} \cdot 13^{T/4}  \right).$$
Suppose $\mu$ is a probability measure on $X$ with $H_{\mu}(\mathcal{P}_{N,T})\geq \rho \log(N)$.  Then for every nonnegative smooth function $f: X \to \C$, there are integers $s,t$ with $$0 \le s \le (1 - \delta)\log_5(N) \quad \text{and} \quad 0 \le t \le T$$ such that
$$(\theta_5^s \theta_{13}^t)_*\mu(f) \ge \left(\rho - 2\delta\right)m(f) - \kappa T^{-\delta/2}\|f\|_{H^{10}}-\kappa N^{-\delta/4} \cdot 13^{T/4} \Lip(f).$$
%
\end{theorem}

We will use this theorem in the form of the following corollary.

\begin{corollary}[Quantitative denseness]\label{cor:quantitativedenseness}
There is a constant $c>0$ such that the following holds.  Let $0<\rho<1$, and let $n \in \mathbb{N}$ be sufficiently large in terms of $\rho$.  Set $\delta:=\rho/10$ and $N:=5^n$ and $T:=\frac{\log 5}{4 \log 13} \cdot \delta n$, and let $\mathcal{P}_{N,T}$ be the partition from Theorem~\ref{thm:entropyweirdpartition}.  Suppose $S \subseteq X$ is a set that intersects at least $N^\rho$ parts of the partition $\mathcal{P}_{N,T}$.  Then the orbit
$$\{\theta_5^s \theta_{13}^t S: 0 \le s,t \le n\}$$
is $n^{-c\rho}$-dense in $X$.
\end{corollary}

\begin{proof}
We will show that for each point $x_0 \in X$, the orbit $\{\theta_5^s \theta_{13}^t S: 0 \le s,t \le n\}$ intersects the ball of radius $n^{-c\rho}$ centered at $x_0$.
For each part $P \in \mathcal{P}_{N,T}$ that intersects $S$, choose some point $x_P \in P \cap S$; let $\mu$ be the (unweighted) average of the point masses at these $x_P$'s.  Thus $H_\mu(\mathcal{P}_{N,T}) \ge \rho \log(N)$.

Let $f: X \to \mathbb{R}_{\geq 0}$ be the bump function around $0$ produced by Lemma~\ref{lem:smoothapproximation} with $u:=\lceil c\rho \log_5 n \rceil+1$, $v:=\lceil c\rho \log_{13} n \rceil+1$, and $\gamma:=n^{-c\rho}/10$.  Fix some $x_0 \in X$, and let $\widetilde f:=f(\cdot-x_0)$ be the shift of $f$ centered at the point $x_0$.  The estimates of Lemma~\ref{lem:smoothapproximation} still apply verbatim and give
$$\|\widetilde{f}\|_{H^{10}} \ll n^{44c\rho} \quad \text{and} \quad \Lip(\widetilde f) \ll n^{c\rho}.$$
Now, applying Theorem~\ref{thm:entropyweirdpartition} to the function $\widetilde f$ provides integers $0 \leq s,t \leq n$ (note that $T \leq n$) such that
\begin{align*}
(\theta_5^s \theta_{13}^t)_*\mu(\widetilde f) &\ge \left(\rho - 2\delta\right)m(\widetilde f) - O(T^{-\delta/2}\|\widetilde f\|_{H^{10}}+N^{-\delta/4} \cdot 13^{T/4} \Lip(\widetilde f))\\
 & \geq 10^{-6}\rho n^{-4c\rho}-O((\rho n)^{-\rho/20}n^{44c\rho}+5^{-n\rho/20} n^{c\rho}).
\end{align*}
This quantity is strictly positive as long as long as $c$ is sufficiently small (e.g., $c=1/10000$ works).  Since $\widetilde f$ is supported in the $n^{-c\rho}$-ball around $x_0$, we conclude that $x_0$ lies within $n^{-c\rho}$ of the orbit $\{\theta_5^s \theta_{13}^t S: 0 \le s,t \le n\}$, as desired.
%
\end{proof}

To establish Theorem~\ref{thm:entropyweirdpartition}, we will establish a discretized version for measures and functions on sets of the form $\gamma+(j \circ \iota^{\Delta})((\theta_5^{-n}A)/A) \subseteq X$; for notational simplicity, we will write $\theta_5^{-n}A/A$ instead of $(j \circ \iota^{\Delta})((\theta_5^{-n}A)/A)$.

The proof has two main ingredients.  The first ingredient regards the behavior ``on average'' of $\xi_* \mu$, where $\mu$ is a probability measure on $\gamma+(\theta_5^{-n}A)/A$ and $\xi$ ranges over $S_{13}(n)$ (recall that $S_{13}(n)$ is the subgroup of $A/\theta_5^n A$ generated by $\theta_{13}$).

\begin{theorem}\label{thm:rigidity}
There is a constant $C>0$ such that the following holds.  Let $\gamma \in X$, and let $\mu$ be a probability measure on $\gamma+\theta_5^{-n}A/A \subseteq X$.  Then for every smooth function $f$ on $X$, we have
$$\E_{\xi \in S_{13}(n)} |(\xi_* \mu)(f) - m(f)|^2 \leq C \|f\|_{H^{10}}^2\|\mu\|_2^2.$$
%
\end{theorem}
\begin{proof}
The first step is showing that all of the nontrivial Fourier coefficients of the measure $\xi_* \mu$ are small on average (over $\xi$).
Recall the isomorphism $\widehat{\theta_5^{-n}A/A} \cong A/\theta_5^n A$ from from Lemma~\ref{lem:dual}, in which we associate $\eta \in A/\theta_5^n A$ with the character
$$\eta(x):= \{\iota_{5}(\eta x)\}_5 + \{\iota_{13}(\eta x)\}_{13} - \Re(\eta x) \pmod{1}.$$
Lemma~\ref{lem:boundedsubgroup} gives us a nonnegative integer $\alpha$ such that $S_{13}(n)$ contains $1+\theta_5^\alpha(A/\theta_5^n A)$ as a subgroup.  Since each multiplicative coset of $1+\theta_5^\alpha(A/\theta_5^n A)$ is an additive coset of $\theta_5^\alpha(A/\theta_5^n A)$, we see that $S_{13}(n)$ is invariant under adding $\theta_5^\alpha$.  It follows that for any $x \in \theta_5^{-n}A/A$ and any character $\eta \in A/\theta_5^n A$, we can write
\begin{align*}
\E_{\xi \in S_{13}(n)} e(\eta(\xi x)) &=e(\eta(\theta_5^\alpha x))\E_{\xi \in S_{13}(n)} e(\eta(\xi x))\\
 &= e(\{\iota_{5}(\theta_5^\alpha\eta x)\}_5 + \{\iota_{13}(\theta_5^\alpha \eta x)\}_{13} - \Re(\theta_5^\alpha \eta x)) \E_{\xi \in S_{13}(n)} e(\eta(\xi x))
\end{align*}
(note that $\xi x \in \theta_5^{-n}A/A$ is well-defined).
Thus the expectation over $\xi$ vanishes whenever $e(\eta(\theta_5^\alpha x)) \neq 1$, and \cite[Proposition 6.2]{Man13} tells us that this occurs precisely when $\theta_5^\alpha \eta x \notin A$.

For each $s \in \theta_5^{-n}A/A$, write $w_s:= \mu(\{\gamma+s\})$.
For each character $\eta \in A/\theta_5^n A$, we have
\begin{align*}
\E_{\xi \in S_{13}(n)} |\widehat{\xi_*\mu}(\eta)|^2 &=\E_{\xi \in S_{13}(n)} \left| \int_X e(\eta(x)) \,d(\xi_* \mu)(x) \right|^2 \\ 
&= \E_{\xi \in S_{13}(n)} \left|\sum_{s \in \theta_5^{-n}A/A}w_s e(\eta(\xi( \gamma+s))) \right|^2 \\
&= \sum_{s, s'} w_s w_{s'} \E_{\xi \in S_{13}(n)} e(\eta(\xi(s - s'))) \\
& \leq \sum_{s, s'} w_s w_{s'} {\bf 1}_{U(\eta)}(s-s'),
\end{align*}
where $U(\eta):=\{t \in \theta_5^{-n}A/A: \theta_5^\alpha\eta t \in A\}$.
Applying the Cauchy--Schwarz Inequality, we can bound the final quantity from the last centered equation as
\begin{align*}
\left(\sum_{s,s'} w_s w_{s'} {\bf 1}_{U(\eta)}(s-s')\right)^2 & \leq \sum_s w_s^2 \cdot \sum_s \left(\sum_{s'} w_{s'} {\bf 1}_{U(\eta)}(s-s')\right)^2\\
 &= \sum_s w_s^2 \cdot \sum_s \sum_{s',s''} w_{s'}w_{s''}{\bf 1}_{U(\eta)}(s-s'){\bf 1}_{U(\eta)}(s-s'')\\
 &=\sum_s w_s^2 \cdot \sum_{s',s''} w_{s'}w_{s''}{\bf 1}_{U(\eta)}(s'-s'')\sum_{s}{\bf 1}_{U(\eta)}(s-s')\\
 &=|U(\eta)|\sum_{s}w_s^2 \cdot \sum_{s',s''} w_{s'} w_{s''} {\bf 1}_{U(\eta)}(s'-s'').
\end{align*}
(The second equality uses the fact that $U$ is a group, and the third equality uses the fact that $\sum_{s}{\bf 1}_{U(\eta)}(s-s')=|U(\eta)|$ for all $s'$.)  Dividing out the common term from the two sides of this inequality and using the count $|U|=5^{\min(\alpha+v_{P_5}(\eta),n)}$, we deduce that
$$\E_{\xi \in S_{13}(n)} |\widehat{\xi_*\mu}(\eta)|^2 \leq \sum_{s,s'} w_s w_{s'} {\bf 1}_{U(\eta)}(s-s') \leq |U(\eta)|\sum_{s}w_s^2= 5^{\min(\alpha+v_{P_5}(\eta),n)} \|\mu\|_2^2.$$
This completes the first step of the proof.

For the second step, Fourier-expand the function $f$ as
$$f(x) = \sum_{\eta \in A} \widehat{f}(\eta)e(\eta(x))$$
(recall that $\widehat X$ is canonically isomorphic to $A$).  Then, by Parseval's Identity and the Cauchy--Schwarz Inequality, we have
\begin{align*}
\E_{\xi \in S_{13}(n)} |(\xi_* \mu)(f) - m(f)|^2 &= \E_{\xi \in S_{13}(n)}\left|\sum_{0 \neq \eta \in A} \widehat{f}(\eta) \widehat{\xi_*\mu}(\eta)\right|^2 \\
&\le \E_{\xi \in S_{13}(n)} \sum_{\eta \neq 0} \frac{|\widehat{\xi_*\mu}(\eta)|^2}{\langle \eta \rangle^{10}} \cdot \sum_{\eta \neq 0} \langle \eta \rangle^{10} |\widehat{f}(\eta)|^2 \\
&\le \|f\|_{H^{10}}^2\|\mu\|_2^2 \sum_{\eta \neq 0} \frac{5^{\min(\alpha+v_{P_5}(\eta),n)}}{\langle \eta \rangle^{10}}.
\end{align*}
It remains to show that the last sum over $\eta$ is of size $O(1)$.

Recall that each nonzero $\eta \in A$ can be expressed uniquely as $\eta=\overline{P_5}^a \overline{P_{13}}^b P_5^c z$, where $a,b \in \mathbb{Z}$, and $c \in \mathbb{N}$, and $z \in \mathbb{Z}[i]$ is coprime to all of $\overline{P_5},\overline{P_{13}},P_5$.  Then
$$\langle \eta \rangle=5^{-a}+13^{-b}+5^{a/2} \cdot 13^{b/2} \cdot 5^{c/2} \cdot |z|.$$
We always have $\langle \eta \rangle \geq 5^{c/4}$: Indeed, if $5^{-a},13^{-b}\leq 5^{c/4}$, then the third term is at least $5^{-c/8} \cdot 5^{-c/8} \cdot 5^{c/2} \cdot |z| \geq 5^{c/4}$ (since $|z| \geq 1$).  Thus our sum of interest can be bounded as
$$\sum_{\eta \neq 0} \frac{5^{\min(\alpha+v_{P_5}(\eta),n)}}{\langle \eta \rangle^{10}} \leq 5^\alpha \sum_{\eta \neq 0} \frac{1}{\langle \eta \rangle^{6}} \ll \sum_{\eta \neq 0} \frac{1}{\langle \eta \rangle^{6}}.$$
Now that we have dispensed with $P_5$, it is more convenient to parameterize $\eta$ as $\eta=\overline{P_5}^a \overline{P_{13}}^b w$, where $a,b \in \mathbb{Z}$ and $w \in \mathbb{Z}[i]$ is coprime to $\overline{P_5}, \overline{P_{13}}$.  We will bound the contribution from each $w$ individually.  Notice that $$\langle \eta \rangle=5^{-a}+13^{-b}+5^{a/2} \cdot 13^{b/2} \cdot |w|\geq \exp(B_w(a,b)),$$ where we have set
$$B_w(a,b):=\max(-a\log 5,-b\log 13,a(\log 5)/2+b(\log 13)/2+\log|w|).$$
Now $B_w$, viewed as a function on $\mathbb{R}^2$, achieves its minimum value $(\log|w|)/2$ at the point $(a_0,b_0)=(-\log_5(|w|)/2,-\log_{13}(|w|)/2)$, and it is easy to see that $B(a,b)-(\log|w|)/2$ grows as at least some absolute positive constant $c$ times the distance from $(a,b)$ to $(a_0,b_0)$.  It follows that
\begin{align*}
\sum_{a,b \in \mathbb{Z}^2} \frac{1}{\langle \overline{P_5}^a \overline{P_{13}}^b w \rangle^6} &\leq \sum_{a,b \in \mathbb{Z}^2} \frac{1}{\exp(6B_w(a,b))}\\
 &\leq \frac{1}{|w|^3} \sum_{a,b \in \mathbb{Z}^2}\frac{1}{\exp(6c \dist((a,b),(a_0,b_0)))}\\
 &\ll \frac{1}{|w|^3} \int_{R=0}^\infty \frac{R}{\exp(6cR)} \,dR \ll \frac{1}{|w|^3}.
\end{align*}
The sum of $|w|^{-3}$ over all nonzero $w \in \mathbb{Z}[i]$ is still $O(1)$. This completes the proof.
\end{proof}

The second main proof ingredient is the following lemma, which says that if $\mathcal{P}_{N,T}$ has large entropy with respect to $\mu$, then there is some small $s$ such that $(\theta_5^s)_* \mu$ majorizes a ``fairly spread-out'' measure $\nu$.  This step is one of the key ideas of~\cite{BLMV08}.

\begin{lemma}\label{lem:partitiontrick}
Let $\rho,\delta>0$ satisfy $\delta \leq \rho/10$.  Let $n \in \mathbb{N}$, and set $N:=5^n$.  Suppose $\mu$ is a measure on $\theta_5^{-n}A/A$ with entropy at least $\rho \log(N)$ (with respect to the partition into singletons).  Then for each integer $10/\delta \leq \ell \leq \delta n$, we have $(\theta_5^s)_*\mu \geq \nu$ for some integer $0 \leq s \leq (1-\delta)n$ and some measure $\nu=\sum_{i=1}^{5^{n-\ell}}w_i \nu_i$ satisfying the following:
\begin{enumerate}
    \item the $\nu_i$'s are probability measures, one supported on each translate of $\theta_5^{-\ell}A/A$;
    \item the $w_i$'s are nonnegative reals satisfying $\sum_i w_i \geq \rho-2\delta$;
    \item we have $\sum_i w_i \|\nu_i\|_2^2 \leq 2\cdot 5^{-\delta \ell}$.
\end{enumerate}
\end{lemma}

\begin{proof}
Lemma~\ref{lem:equivalence} provides an isomorphism $\theta_5^{-n}A/A \cong \mathbb{Z}/5^n \mathbb{Z}$ that intertwines the $\times \theta_5$ and $\times 5$ maps.  After we apply this isomorphism, the desired result is precisely the content of \cite[Lemma 3.6]{BLMV08}
\end{proof}

With these two ingredients in hand, we can prove a ``discretized'' version of Theorem~\ref{thm:entropyweirdpartition} in which $\mu$ is supported on $\theta_5^{-n}A/A$; compare with~\cite[Proposition 3.8]{BLMV08}.

\begin{proposition}\label{cor:discreteblmv}
There is a constant $\kappa>0$ such that the following holds.  Set $\gamma:=\log 5/(4\log 13)$.  Let $\rho,\delta>0$ and $n,T \in \mathbb{N}$ (with $N:=5^n$) be such that $$10/\log_5 T \le \delta \le \rho/10 \quad \text{and} \quad 5^{20/\delta} \le T \le \gamma \delta n.$$
Suppose $\mu$ is a measure on $\theta_5^{-n}A/A$ with entropy at least $\rho \log(N)$ (with respect to the partition into singletons).  Then for every nonnegative smooth function $f$ on $X$, there are integers $$0 \leq s \leq (1-\delta)n \quad \text{and} \quad 0 \leq t \leq T$$
such that
$$(\theta_5^s \theta_{13}^t)_*\mu(f) \ge (\rho - 2\delta)m(f) - \kappa T^{-\delta/2} \|f\|_{H^{10}}.$$
\end{proposition}
We remark that the proof requires the choice of $s,t$ to depend on $f$.

\begin{proof}
Set $\ell:=\lfloor \log_5(T) \rfloor$.  With a view towards applying Lemma~\ref{lem:partitiontrick}, we claim that $10/\delta\leq \ell \leq \delta n$.  The lower bound is clear.  For the upper bound, notice that $5^{20/\delta} \cdot (\log 13)^2 \geq 1/8$ (with room to spare) and hence
$$\delta n \cdot \frac{\delta n}{T} \geq \frac{5^{20/\delta}}{\gamma} \cdot \frac{1}{\gamma} \geq \frac{2}{(\log 5)^2}.$$
Since $x^2/2 \leq e^x$ for $x \geq 1$, we obtain
$$T \leq \frac{((\log 5) \delta n)^2}{2} \leq e^{(\log 5) \delta n}=5^{\delta n},$$
and on taking logarithms we obtain $\ell \leq \log_5 T \leq \delta n$, as desired.  We now apply Lemma~\ref{lem:partitiontrick} to obtain a measure $\nu=\sum_{i=1}^{5^{n-\ell}} w_i \nu_i$ and an integer $0 \leq s \leq (1-\delta)n$ as in the conclusion of that lemma.  Set $w:=\sum_i w_i \geq \rho-2\delta$.

Deploying our two main ingredients from above and letting $C>0$ be the constant from Theorem~\ref{thm:rigidity}, we find that
\begin{align*}
\E_{\xi \in S_{13}(n)} |(\xi_* \nu)(f) - wm(f)| &\leq \sum_i w_i \E_{\xi \in S_{13}(\ell)} |(\xi_* \nu_i)(f) - m(f)|\\
 &\leq w^{1/2} \left (\sum_i w_i \left (\E_{\xi \in S_{13}(\ell)} |(\xi_* \nu_i)(f) - m(f)| \right)^2 \right)^{1/2}\\
 & \leq \left (\sum_i w_i \E_{\xi \in S_{13}(\ell)} |(\xi_* \nu_i)(f) - m(f)|^2 \right)^{1/2}\\
 &\leq \left (\sum_i w_i C\|f\|_{H^{10}}^2 \|\nu_i\|_2^2 \right)^{1/2}\\
 & \leq C^{1/2} \|f\|_{H^{10}} \left(2 \cdot 5^{-\delta \ell} \right)^{1/2}.
\end{align*}
(In these inequalities we have used the Triangle Inequality; the Cauchy--Schwarz Inequality; the Cauchy-Schwarz Inequality together with $w \leq 1$; Theorem~\ref{thm:rigidity}; and Lemma~\ref{lem:partitiontrick}, respectively.)  It follows that there is some $\xi=13^t$, with $$0 \leq t \leq |S_{13}(\ell)| \leq 5^\ell \leq T,$$
such that
$$|(\theta_{13}^t)_* \nu(f) - wm(f)| \leq (2C)^{1/2} \|f\|_{H^{10}} \cdot 5^{-\delta \ell/2}.$$
Combining this inequality with $(\theta_5^s)_* \mu \geq \nu$, we conclude that
$$(\theta_5^s \theta_{13}^t)_* \mu(f) \geq (\rho-2\delta)m(f)-(2C)^{1/2} \|f\|_{H^{10}} \cdot 5^{-\delta \ell/2} \geq (\rho-2\delta)m(f)-\kappa \|f\|_{H^{10}} T^{-\delta/2},$$
for a suitable choice of $\kappa>0$.
\end{proof}

We can finally deduce Theorem~\ref{thm:entropyweirdpartition} from its discretized version.

\begin{proof}[Proof of Theorem~\ref{thm:entropyweirdpartition}]
Recall that the parts of the partition $\mathcal{P}_{N,T}$ are indexed by the elements of $\theta_5^{-n}A/A$ (the fundamental domain $F$ from Proposition~\ref{prop:constructionofpartition} contains $0$ but no other element of $\theta_5^{-n}A/A$).  Define the probability measure $\mu'$ on $\theta_5^{-n}A/A$ via $$\mu'(\{a\}):= \mu(P_a),$$ where $P_a$ is the part of $\mathcal{P}_{N,T}$ containing $a$.  Thus the entropy of $\mu'$ (with respect to the partition into singletons) is precisely $H_\mu(\mathcal{P}_{N,T})$.  Applying Proposition~\ref{cor:discreteblmv} to $\mu'$, we obtain integers $0 \le s \le (1 - \delta)\lfloor \log_5(N) \rfloor$ and $0 \le t \le T$ such that
$$(\theta_5^s\theta_{13}^t)_*\mu'(f) \ge (\rho - 2\delta)m(f) - \kappa T^{-\delta/2}\|f\|_{H^{10}},$$
where $\kappa>0$ is the constant from Proposition~\ref{cor:discreteblmv}.

We now make use of the ``skewed'' shape of the parts of the partition $\mathcal{P}_{N,T}$.  Recall that each part of the partition $\mathcal{P}_{N,T}$ has diameters at most $N_5, N_{13}, N_{\C}$ in the $\Q_5,\Q_{13},\C$-directions (respectively).
The choice of $N_5, N_{13}, N_{\C}$ and the upper bounds on $s, t$ guarantee that
$$\dist(\theta_5^s\theta_{13}^tx, \theta_5^s\theta_{13}^t y) \leq 5^s N_5+13^t N_{13}+N_{\C} \ll N^{-\delta/4} \cdot 13^{T/4}$$
whenever $x,y$ are points in the same part of the partition $\mathcal{P}_{N,T}$.
Thus
\begin{align*}
|(\theta_5^s\theta_{13}^t)_*\mu(f) - (\theta_5^s\theta_{13}^t)_*\mu'(f)| &\le \sup_{\text{$x,y$ in same part of $\mathcal{P}_{N,T}$}} |f(\theta_5^s\theta_{13}^t x) - f(\theta_5^s\theta_{13}^t y)|\\
 &\ll N^{-\delta/4} \cdot 13^{T/4} \Lip(f),
\end{align*}
and the theorem follows.
%
\end{proof}

\subsection{The main argument}
We now combine Theorem~\ref{thm:mannersapproximation} and Corollary~\ref{cor:quantitativedenseness} from the previous two subsections.  We start with the hypothesis that $y=(\theta_5^s\theta_{13}^t-\theta_5^{s'}\theta_{13}^{t'})p$ is fairly small but not extremely close to $B^{\C}_{1/10}(0)$.  Since $y$ must be reasonably large in either the $\Q_5$-component or the $\Q_{13}$-component, Theorem~\ref{thm:mannersapproximation} tells us that that the orbit of $y$ is quite dense in $X$.  We then deduce that the orbit of $p$ intersects many ($N^{0.49}$, say) parts of the partition $\mathcal{P}_{N,T}$ from Corollary~\ref{cor:quantitativedenseness}.   Finally, this corollary tells us that a slightly longer orbit of $p$ is $\varepsilon$-dense in $X$, and in particular that the orbit intersects $E_1^*(\varepsilon)$.  The details are as follows.

\begin{proposition}[Minor arcs]\label{prop:minor}
There is a constant $C>0$ such that the following holds.  Let $0<\varepsilon, \varpi<1/10$, and let $N_0 \in \N$.  Set $\eta:=5^{-\varepsilon^{-C}}$.  Suppose $p \in X$ and $0 \leq s_1,s_2,t_1,t_2 \leq N_0$ are such that $(\theta_5^{s_1}\theta_{13}^{t_1}-\theta_5^{s_2}\theta_{13}^{t_2})p=j(a,b,z)$ with
$$2\varpi \leq |a|_5+|b|_{13} \leq 13^{-C \eta^{-5}} \quad \text{and} \quad |z| \leq \eta/10.$$
Then the orbit
$$\Theta_{N_0+\log_5(\varpi^{-1})+C\eta^{-5}} \cdot p$$
is $\varepsilon$-dense in $X$ and, in particular, intersects $E_1^*(\varepsilon)$.
\end{proposition}

\begin{proof}
Consider the case $|a|_5 \geq |b|_{13}$; the case $|a|_5 \leq |b|_{13}$ is identical, and we will omit it.  Thus $\varpi \leq |a|_5 \leq 13^{-C\eta^{-5}} \leq \eta/10$ (the last inequality with much extra room), and Theorem~\ref{thm:mannersapproximation} applied to the point $(\theta_5^{s_1}\theta_{13}^{t_1}-\theta_5^{s_2}\theta_{13}^{t_2})p$ tells us that the orbit
$$\{\theta_5^s \theta_{13}^t (\theta_5^{s_1}\theta_{13}^{t_1}-\theta_5^{s_2}\theta_{13}^{t_2})p: 0 \leq s \leq \log_5(\varpi^{-1} \eta^{-4})+C, ~ 0 \leq t \leq C\eta^{-5}\}$$
is $\eta$-dense in $X$ as long as $C$ is sufficiently large.  Notice that this orbit is contained in the set $Z-Z$, where
$$Z:=\Theta_{N_0+\log_5(\varpi^{-1})+C\eta^{-5}} \cdot p=\{\theta_5^s \theta_{13}^t p: 0 \leq s, t \leq N_0+\log_5(\varpi^{-1})+C\eta^{-5}\}$$
is an orbit of our initial point $p$ (here taking a slightly larger value of $C$).

Let $N:=5^n$ be a parameter to be determined shortly.  With the choice $\rho:=0.49$, set $\delta:=\rho/10$ and $T:=\frac{\log 5}{40 \log 13} \cdot \rho n$ as in Corollary~\ref{cor:quantitativedenseness}.  The construction of Proposition~\ref{prop:constructionofpartition} ensures that in the partition $\mathcal{P}_{N,T}$ from Theorem~\ref{thm:entropyweirdpartition}, each part has in-radius
\begin{align*}
\asymp \min(N_5,N_{13},N_{\C}) &=\min\left(\frac{13^{T/4}}{N^{1+3\delta/4}}, \frac{1}{N^{\delta/4} \cdot 13^{3T/4}}, \sqrt{130} \cdot N^{\delta/2} \cdot 13^{T/4}  \right)\\
 &\asymp \min \left(\frac{1}{N^{1+11\delta/16}}, \frac{1}{N^{7\delta/17}}, N^{9\delta/16} \right)\\
 &=\frac{1}{N^{1+11\delta/16}}>\frac{1}{N^{1.1}}.
 \end{align*}
Thus, if we choose $N:=\eta^{-0.9}$ (say), then we find that $Z-Z$
intersects all of the parts of $\mathcal{P}_{N,T}$.  Since the difference set of any two parts of $\mathcal{P}_{N,T}$ intersects only $O(1)$ parts of $\mathcal{P}_{N,T}$ (by Proposition~\ref{prop:constructionofpartition}), it must be the case that $Z$ intersects $\gg |\mathcal{P}_{N,T}|^{1/2}=N^{1/2}$ parts of $\mathcal{P}_{N,T}$; in particular, $Z$ intersects at least $N^{0.49}$ parts of $\mathcal{P}_{N,T}$ since $N$ is sufficiently large.

We can now apply Corollary~\ref{cor:quantitativedenseness}, which tells us that
$$\{\theta_5^s \theta_{13}^t Z: 0 \leq s,t \leq n\}=\Theta_{N_0+\log_5(\varpi^{-1})+C\eta^{-5}+n} \cdot p$$
is $n^{-c\rho}$-dense in $X$ for some absolute constant $c>0$.  Unraveling the definitions of $n,\eta$, we have
$$n^{-c\rho}=(\log_5(N))^{-c\rho}=(0.9\log_5(\eta^{-1}))^{-c\rho}=(0.9\varepsilon^{-C})^{-c\rho}<\varepsilon,$$
as long as $C$ is sufficiently large relative to $c$.  To conclude, notice that in the upper bound on $s,t$, we can absorb the $n=\eta^{-0.9}$ term into the $C\eta^{-5}$ term by increasing the value of $C$.
\end{proof}


\section{Deduction of the quantitative rationality lemma}\label{s:maindeduction}
We can finally assemble all of the pieces.  

\begin{proof}[Proof of Theorem~\ref{thm:quantitative-rationality}]
Let $p \in X$ be any point.  Our goal is to show that either $p$ is very close to a complex ball of radius $1/2$ around a torsion point of small order, or a reasonably short orbit of $p$ under $\theta_5, \theta_{13}$ intersects $E^*(\varepsilon)$.

Let $N_0:=\exp\exp(\varepsilon^{-C_1})$ for some $C_1>0$ to be chosen later. 
 First, consider the orbit
$$\Theta_{N_0} \cdot p=\{\theta_5^s \theta_{13}^t p: 0 \leq s,t \leq N_0\}.$$
The space $X$ can be covered by $O(\gamma^{-4})$ balls of radius $\gamma$ for any $0<\gamma<1$ (recall that $\Z_5 \times \Z_{13} \times [-1/2,1/2) \times [-1/2,1/2)i$ is a fundamental domain for $X$), so the pigeonhole principle provides some $0 \leq s_1,s_2,t_1,t_2 \leq N_0$, with $(s_1,t_2) \neq (s_2,t_2)$, such that $$\dist((\theta_5^{s_1}\theta_{13}^{t_1}-\theta_5^{s_2}\theta_{13}^{t_2})p,0)=\dist(\theta_5^{s_1} \theta_{13}^{t_1}p, \theta_5^{s_2} \theta_{13}^{t_2}p) \ll N_0^{-1/2}.$$
Set $r:=\theta_5^{s_1}\theta_{13}^{t_1}-\theta_5^{s_2}\theta_{13}^{t_2}$, and notice that $r$ is a nonzero Gaussian rational of height at most $\exp(100N_0)$. We can write $rp=j(a,b,z)$, where $|a|_5+|b|_{13}+|z| \ll N_0^{-1/2}$.  We now have a dichomety according to the size of $|a|_5+|b|_{13}$.   

{\bf Major arcs.}  First, suppose that $$|a|_5+|b|_{13} \leq \min(\delta \exp(-200 N_0), \exp(-\exp(200C_2N_0))),$$ where $C_2>0$ is the constant $C$ from Proposition~\ref{prop:super-major}.  Then this proposition (applied with $k:=\exp(-100N_0)$) implies that either $\dist(p, \overline{B}_{1/2}^{\C}(q)) \leq \delta$ for some torsion point $q \in X$ of order at most $\exp(200N_0)$, or $\Theta_{\exp(200C_2 N_0)} \cdot p$ intersects $E^*_1(\varepsilon)$.  Both outcomes are acceptable once we note that
$$\exp(200N_0),\exp(200C_2N_0) \leq \exp\exp \exp(\varepsilon^{-C})$$
for $C$ sufficiently large in terms of $C_1,C_2$.

{\bf Minor arcs.}  Second, suppose that
$$|a|_5+|b|_{13} \geq \min(\delta \exp(-200 N_0), \exp(-\exp(200C_2N_0))),$$
with the same value of $C_2$ as above.  Let $C_3$ denote the constant $C$ from Proposition~\ref{prop:minor}, and set $\eta:=5^{-\varepsilon^{-C_3}}$.  Notice that
$$|a|_5+|b|_{13}+|z| \ll N_0^{-1/2}$$
is smaller than $13^{-C_3 \eta^{-5}}$ as long as $C_1$ is sufficiently large relative to $C_3$.  Thus, applying Proposition~\ref{prop:minor} with $2\varpi:=\min(\delta \exp(-200 N_0), \exp(-\exp(200C_2N_0)))$, we conclude that
$$\Theta_{N_0+\log_5(\varpi^{-1})+C_3 \eta^{-5}} \cdot p$$
intersects $E^*_1(\varepsilon)$.  This outcome is acceptable once we note that
\begin{align*}
N_0+\log_5(\varpi^{-1})+C_3 \eta^{-5} &\ll N_0+\max(\log(\delta^{-1})+N_0,\exp(200 C_2 N_0))+5^{5\varepsilon^{-C_3}}
\end{align*}
is smaller than $$\max(\exp\exp\exp(\varepsilon^{-C}), C\log(\delta^{-1}))$$
for $C$ sufficiently large in terms of $C_1,C_2,C_3$.  This concludes the proof.
\end{proof}


\section{From the quantitative rationality lemma to the main theorem}\label{s:irrationalargument}

In this section we deduce Theorem~\ref{thm:main} (the main result about the Pyjama Problem) from Theorem~\ref{thm:quantitative-rationality} (the quantitative rationality lemma).  This part of the argument closely follows Manners's deduction of his main theorem from his qualitative rationality lemma (see Sections 3 and 5.3 of \cite{Man13}).  The only real point of difference lies in Lemma~\ref{lem:irrationality-trick} below.  Like Manners, we separate the deduction into two steps.

\subsection{Lattice point statement}
In the first step, we convert Theorem~\ref{thm:quantitative-rationality} into a statement about the complement of the rotated strips $E_\theta(\varepsilon)$; compare~\cite[Lemma 3.1]{Man13}.  We isolate the following elementary fact, proven as the ``Claim'' in~\cite[page 251]{Man13}: Let $0<\eta<1/100$ and $B \in \mathbb{N}$; if $z \in \mathbb{C}$ is such that $\theta_5^s\theta_{13}^tz$ lies within distance $\eta$ of $\mathbb{Z}[i]$ for all $0 \leq s,t \leq B$, then in fact $z$ lies within $\eta$ of $\overline{P}_5^B \overline{P}_{13}^B \mathbb{Z}[i]$.  

For $m$ a natural number and $D$ a (nonzero) Gaussian integer, define the subset
$$F(m,D):=\left(\frac{D}{m}\mathbb{Z}[i] \setminus D\mathbb{Z}[i]\right)+\overline{B}_{2/3}(0) \subseteq \mathbb{C};$$
as usual, the choice of the constant $2/3$ is not important (anything strictly larger than $1/2$ works).

\begin{lemma}\label{lem:lattice-points}
There is an absolute constant $C>0$ such that for all $0<\varepsilon<1/10$ and $R>\exp\exp(\varepsilon^{-C})$, the following holds with $$n=n_{\varepsilon}:= \exp\exp\exp(\varepsilon^{-C})$$ and $$N=N_{\varepsilon,R}:=\max(\exp\exp\exp(\varepsilon^{-C}),\exp\exp(\varepsilon^{-C})+3\log R).$$ 
There is a Gaussian integer $D$ of modulus at least $R$ such that
$$\bigcup_{1 \leq m \leq n}F(m,D) \cup \bigcup_{\theta \in \Theta_N}E_\theta(\varepsilon)=\mathbb{C}.$$
\end{lemma}

\begin{proof}
Theorem~\ref{thm:quantitative-rationality}, applied with $\delta:=1/(1000nR^3)$,\footnote{Recall that $n$ depends only on $\varepsilon$, so we are free to choose $\delta$ after seeing the value of $n$.} provides the desired $n,N$ with the property that
$$\bigcup_{x \in X: ~\ord(x) \leq n} \left(\overline{B}^{\mathbb{C}}_{1/2}(x)+B_{\delta}(0)\right) \cup \bigcup_{\theta \in \Theta_N} E_\theta^*(\varepsilon)=X.$$
Applying $j_{\mathbb{C}}^{-1}$ to both sides, we get
$$\bigcup_{x \in X: ~\ord(x) \leq n} j_{\mathbb{C}}^{-1} \left(\overline{B}^{\mathbb{C}}_{1/2}(x)+B_{\delta}(0)\right) \cup \bigcup_{\theta \in \Theta_N} E_\theta(\varepsilon)=\mathbb{C};$$
it remains to show that
$$\bigcup_{x \in X: ~\ord(x) \leq n} j_{\mathbb{C}}^{-1} \left(\overline{B}^{\mathbb{C}}_{1/2}(x)+B_{\delta}(0)\right) \setminus \bigcup_{\theta \in \Theta_N} E_\theta(\varepsilon) \subseteq \bigcup_{1 \leq m \leq n}F(m,D).$$

Let $w$ be an element of the left-hand side.  Then we can write $$w=j_{\mathbb{C}}^{-1}(x+y+z),$$ where $x \in X$ is a torsion point of some order $m \leq n$; $y=j_{\mathbb{C}}(y')$ for some complex number $y'$ of size at most $1/2$; and $z \in X$ has size at most $\delta$.  Applying $j_{\mathbb{C}}$, we get
$$j_{\mathbb{C}}(w)=x+j_{\mathbb{C}}(y')+z.$$
Subtracting $j_{\mathbb{C}}(y')$ from both sides and then multiplying by $m$ gives
$$j_{\mathbb{C}}(m(w-y'))=mx+mz=mz.$$
Evaluating both sides at $\theta_5^s \theta_{13}^t$ (for $s,t$ nonnegative integers) gives that the distance from the real part of $m(w-y')\theta_5^s \theta_{13}^t$ to the nearest integer is at most
$$m|z(\theta_5^s \theta_{13}^t)|\leq n\delta(5^s+13^t+1).$$
Likewise, evaluating at $i\theta_5^s \theta_{13}^t$ shows that the imaginary part of $m(w-y')\theta_5^s \theta_{13}^t$ is within $n\delta(5^s+13^t+1)$ of an integer, so we conclude that $m(w-y')\theta_5^s \theta_{13}^t$ lies within distance $2n\delta(5^s+13^t+1)$ of a Gaussian integer.

Set $D:=\overline{P_5}^{\lfloor \log R\rfloor} \overline{P_{13}}^{\lfloor \log R\rfloor}$; notice that $|D|>R$ (with room to spare).  For each choice of $s,t \leq \lfloor \log R\rfloor$, the point $m(w-y')\theta_5^s \theta_{13}^t$ lies within $$2n\delta(5^s+13^t+1) \leq 10n\delta R^3\leq 1/100$$ of a Gaussian integer.  The fact mentioned before the proof implies that $m(w-y')$ lies within $10n\delta R^3$ of $D\mathbb{Z}[i]$.  Dividing by $m$ and subtracting $y'$ (which, recall, has modulus at most $1/2$), we find that $w$ lies within $10n\delta R^3+1/2 \leq 2/3$ of $\frac{D}{m}\mathbb{Z}[i]$.

To conclude that $w \in F(m,D)$, we must show that $w$ does not lie within distance $2/3$ of any element of $D\mathbb{Z}[i]$.  Suppose for the sake of contradiction that $w$ does lie within distance $2/3$ of an element of $D\mathbb{Z}[i]$, and write $w=x+y$ with $x \in D\mathbb{Z}[i]$ and $|y| \leq 2/3$.  We will show that there is some $0 \leq s \leq \varepsilon^{-C}$ (where $\varepsilon^{-C}$ is much smaller than $N$) such that $$\Re(\theta_5^s w) \in (-\varepsilon,\varepsilon) \pmod{1},$$ i.e., $w \in E_{\theta_5^s}(\varepsilon)$, contrary to our initial assumption that $$w \notin \bigcup_{\theta \in \Theta_N}E_\theta(\varepsilon).$$
Since $\theta_5^s=P_5^s/\overline{P_5}^s$ is a Gaussian rational with denominator dividing $D$, we see that $\theta_5^s x \in \mathbb{Z}[i]$ and hence $\Re(\theta_5^s x) \in \mathbb{Z}$.  So it suffices to find $s$ with $\Re(\theta_5^s y) \in (-\varepsilon,\varepsilon) \pmod{1}$.

Lemma~\ref{lem:quantitativeapproximation} with $K:=\varepsilon^{-C}$ tells us that $\{\theta_5^s: 0 \leq s \leq \varepsilon^{-C}\}$ is $\varepsilon^{cC}$-dense in the unit circle for some absolute constant $c>0$.  We can ensure that $C>2/c$, so that $\{\theta_5^s: 0 \leq s \leq \varepsilon^{-C}\}$ is (say) $\varepsilon/10$-dense.  It follows that $\{\Re(\theta_5^s y): 0 \leq s \leq \varepsilon^{-C}\}$ is $\varepsilon/10$-dense in the interval $[-|y|, |y|]$ (recall that $|y| \leq 2/3$), and in particular there is some $0 \leq s \leq \varepsilon^{-C}$ such that $\Re(\theta_5^s y) \in (-\varepsilon, \varepsilon)$.  This completes the proof.
\end{proof}

\subsection{A few more rotations}

The last step is enlarging the set $\Theta_N$ of rotations in order to fill in the ``gaps'' $\cup_{1 \leq m \leq n}F(m,D)$ from Lemma~\ref{lem:lattice-points}.  Manners achieves this using what he calls the ``irrational trick''.  This trick does not directly apply in our setting, however, since Manners's analogue of Lemma~\ref{lem:lattice-points} has only a single set $F(m,D)$ rather than several such sets.  We could reduce to Maners's setting by noting that $\cup_{1 \leq m \leq n}F(m,D) \subseteq F(n!,D)$, but this would be too wasteful quantitatively (it would cost us an additional exponential in the main theorem).  Instead, we will use a geometric argument that generalizes the idea behind the irrationality trick.

For $m \in \mathbb{N}$, set $\alpha_m:=\arccos(1/(2m))/(2\pi)$, so that
$e(\alpha_m)+e(-\alpha_m)=1/m$.  For $n \in \mathbb{N}$, define the subset
$$V(n):=\{e(\pm \alpha_m): 1\leq m \leq n\} \cup \{1\}$$
of the complex unit circle; notice that $|V(n)|=2n+1$.

\begin{lemma}\label{lem:irrationality-trick}
Let $n \in \mathbb{N}$.  Then for every partition $V_1 \sqcup \cdots \sqcup V_n$ of $V(n)$, there are integers $\lambda_v$ (for $v \in V(n) \setminus \{1\}$) such that the following holds:
\begin{itemize}
    \item $\sum_{v \in V \setminus \{1\}} \lambda_v v=1$;
    \item if $v \in V_m$, then $\lambda_v$ is a multiple of $m$;
    \item $|\lambda_v| \leq n^{8}$ for all $v \in V \setminus \{1\}$.
\end{itemize}
\end{lemma}

\begin{proof}
First, consider a single $1 \leq m \leq n$.  Suppose $e(\alpha_m) \in V_{m_1}$ and $e(-\alpha_m) \in V_{m_2}$.  Recalling $e(\alpha_m)+e(-\alpha_m)=1/m$, consider the quantity
\begin{equation*}\label{eq:lcm-multiple}
\beta_m:=\frac{m}{\gcd(m,\lcm(m_1,m_2))} \cdot \lcm(m_1,m_2)(e(\alpha_m)+e(-\alpha_m))=\frac{\lcm(m_1,m_2)}{\gcd(m,\lcm(m_1,m_2))}.
\end{equation*}
From the first expression we see that $\beta_m$ is an integer linear combination of $m_1e(\alpha_m)$, $m_2e(-\alpha_m)$, where the coefficients of $e(\alpha_m), e(-\alpha_m)$ have size at most $n^3$.  From the right-hand side we see that $\beta_m$ is a natural number of size at most $n^2$.

We claim that the natural numbers $\beta_m$, for $1 \leq m \leq n$, have total gcd $1$.  To this end, it suffices to show that for each prime $p$, there is some $m$ such that $\beta_m$ is coprime to $p$.  Fix a prime $p$, and consider $\beta_m$ for $m=p^{\lfloor \log_p(n) \rfloor}$.  Then certainly $v_p(\lcm(m_1,m_2)) \leq \lfloor \log_p(n) \rfloor$, so $v_p(\gcd(m,\lcm(m_1,m_2)))=v_p(\lcm(m_1,m_2))$, and $\beta_m$ is not divisible by $p$.  This proves the claim.

It follows that there are integers $\gamma_1, \ldots, \gamma_n$ such that
\begin{equation}\label{eq:coprime}
\sum_{m=1}^n \gamma_m \beta_m=1;
\end{equation}
among all such choices, take one that minimizes $\sum_m |\gamma_m|$.  Notice that there cannot be indices $k,\ell$ with $\gamma_k\geq n^2$ and $\gamma_\ell\leq -n^2$, since then replacing $(\gamma_k,\gamma_\ell)$ with $(\gamma_k-\beta_\ell, \gamma_\ell+\beta_k)$ would give a choice of the $\gamma_m$'s with a smaller value of $\sum |\gamma_m|$ (recall that $1 \leq \beta_k,\beta_\ell \leq n^2$).  Thus, either $\gamma_m \geq -n^2$ for all $m$ or $\gamma_m \leq n^2$ for all $m$.  If the former occurs, then the sum of the positive $\gamma_m$'s is at most $(n-1) \cdot n^2 \cdot n^2+1 \leq n^5$, and we conclude that $|\gamma_m| \leq n^5$ for all $m$; we obtain the same conclusion in the latter case.

Unraveling the definitions of $\beta_m,\gamma_m$, we see that \eqref{eq:coprime} provides the necessary linear combination of the $v$'s.  Explicitly, for $v=e(\pm \alpha_m)$, we can set
$$\lambda_v:=\frac{\gamma_m m \lcm(m_1,m_2)}{\gcd(m,\lcm(m_1,m_2))},$$
with $m_1,m_2$ as in the definition of $\beta_m$ from the first paragraph.
\end{proof}

The deduction of Theorem~\ref{thm:main} is now quick.

\begin{proof}[Proof of Theorem~\ref{thm:main}]
Apply Lemma~\ref{lem:lattice-points} with $R:=n^{10}+1$ to obtain $n,N,D$ as in the conclusion of that lemma.\footnote{Again, note that Lemma~\ref{lem:lattice-points} gives us $n$ depending only on $\varepsilon$, at which point we may choose $R$ depending on $n$.}  This lets us take
$$n=\exp\exp\exp(\varepsilon^{-C}) \quad \text{and} \quad N=\exp\exp\exp(\varepsilon^{-C})$$
for some constant $C>0$.
With this value of $n$, set $V:=V(n)$, and define the set of rotations
$$\Theta:=\Theta_NV=\{\psi v: \psi \in \Theta_N, v \in V\}.$$
Notice that
$$|\Theta|\leq |\Theta_N| \cdot |V|\leq \exp\exp\exp(\varepsilon^{-2C})$$
(say).
We claim that
$$\bigcup_{\theta \in \Theta}E_{\theta}(\varepsilon)=\mathbb{C}.$$
Suppose for the sake of contradiction that there is some
$$z \in \mathbb{C} \setminus \bigcup_{\theta \in \Theta}E_{\theta}(\varepsilon).$$
Then for each $v \in \Theta$, the complex number $vz$ lies in $\cup_{1 \leq m \leq n}F(m,D)$.  Let $$V_m:=\{v \in V: vz \in F(m,D)\}.$$  Then clearly $V_1 \cup \cdots \cup V_n=V$, and Lemma~\ref{lem:irrationality-trick} produces some coefficients $\lambda_v$ satisfying the conclusions of that lemma. 
 Multiplying through the first conclusion by $z$, we obtain
$$\sum_{v \in V \setminus \{1\}} \lambda_v vz=z.$$
The left-hand side lies within $(2n/3)\sum_v |\lambda_v| \leq n^{10}$ of $D\mathbb{Z}[i]$, which contradicts our assumption on $z$ since $|D| \geq R>n^{10}$.
\end{proof}

\section{Concluding remarks}\label{sec:concluding}

\subsection{Three exponentials}
Our main result Theorem~\ref{thm:main} shows that it is possible to cover the whole plane with $\exp \exp \exp(\varepsilon^{-O(1)})$ rotations of the pyjama stripe $E(\varepsilon)$.  This bound is three exponentials worse than the polynomial bounds that one might hope are possible, and we can correspondingly point to three places in our proof where we ``lose'' exponentials.  We lose one exponential in our application of Baker's Theorem in Section~\ref{sec:baker}.  We lose another exponential in Theorem~\ref{thm:mannersapproximation}, Manners's denseness criterion.  And we lose a third exponential in the entropy arguments of Section~\ref{subs:entropy}.

One might hope to save at least the third (and perhaps second) of these exponentials by using a growing number of Gaussian primes, rather than just $P_5, P_{13}$, and applying more efficient Type I exponential sum estimates for smooth numbers as in \cite{DS24}. Saving the exponent coming from Baker's theorem, however, seems out of reach of current methods. It is plausible if we had very strong results on gaps between smooth numbers (such as quasi-polynomial bounds on gaps between $(\log x)^{1 - \varepsilon}$-smooth numbers up to $x$), then such a result could replace Baker's Theorem and let us eliminate the last exponential.

\subsection{Lonely runners wearing pyjamas}

To our initial surprise, a slight variant of the Pyjama Problem is equivalent to the Lonely Runner Problem of Wills~\cite{wills1967zwei} and Cusick~\cite{cusick1972view}.  Since this connection seems not to have been previously observed, we sketch it here.

Suppose that we allow ourselves to dilate the closure $\overline{E}(\varepsilon)$ of the pyjama stripe  as well as rotate it; this amounts to multiplying $\overline{E}(\varepsilon)$ by \emph{any} nonzero complex number, rather than just any unit-norm complex number.  Let $\alpha(N)$ denote the minimum $\varepsilon$ such that it is possible to cover the whole plane with $N$ rotated and dilated copies of $\overline{E}(\varepsilon)$.  (In Theorem~\ref{thm:main} we took the ``dual'' perspective of fixing $\varepsilon$ and asking for the smallest number $N$ of rotations of $E(\varepsilon)$ that can cover the plane.)

From the Lonely Runner side, define the \emph{maximum loneliness} of a set of nonzero reals $v_1, \ldots, v_N$  to be $$\ML(v_1, \ldots, v_N):=\sup_{t \in \mathbb{R}} \min_{1 \leq j \leq N} \|tv_j\|_{\mathbb{R}/\mathbb{Z}}.$$
Let $\beta(N)$ denote the minimum value of $\ML(v_1, \ldots, v_N)$ for $v_1, \ldots, v_N$ nonzero.  In this language, the Lonely Runner Conjecture asserts that $\beta(N)=1/(N+1)$; the trivial volume lower bound is $\beta(N) \geq 1/(2N)$, and the example $(v_1, \ldots, v_N)=(1,\ldots, N)$ shows that $\beta(N) \leq 1/(N+1)$.  It is known that $\beta(N)$ is attained by some choice of natural numbers $v_1, \ldots, v_N$.

The connection between these two problems is that in fact $\alpha(N)=\beta(N)$ for all $N$.  To see this, it is helpful to notice that $\ML(v_1, \ldots, v_N)$ is the smallest $\varepsilon$ such that the sets $$\{t \in \mathbb{R}: tv_i \in [-\varepsilon,\varepsilon] \pmod{1}\}$$ (for $1 \leq i \leq N$) together cover all of $\mathbb{R}$.  We will show that $\alpha(N) \leq \beta(N)$ and $\alpha(N) \geq \beta(N)$.

For the first inequality, we claim that for any nonzero $v_1, \ldots, v_N$, the dilations of the pyjama stripe $\overline{E}(\ML(v_1, \ldots, v_N))$ by $1/v_1, \ldots, 1/v_N$ (with no rotations) together cover the whole plane.  Indeed, the dilation of $\overline{E}(\ML(v_1, \ldots, v_N))$ by $1/v_j$ is precisely
    $$\{(x+iy): xv_j \in [-\ML(v_1, \ldots, v_N),\ML(v_1, \ldots, v_N)] \pmod{1}\},$$
    and the definition of $\ML(v_1, \ldots, v_N)$ guarantees that these sets together cover the whole plane.  Thus $\alpha(N) \leq \ML(v_1, \ldots, v_N)$ for all choices of $v_1, \ldots, v_N$, and we conclude that $\alpha(N) \leq \beta(N)$.

For the second inequality, suppose we have $N$ rotated and dilated copies of $\overline{E}(\varepsilon)$ that together cover the whole plane; without loss of generality, we may assume that none of the rotations is by exactly $\pi/2$, so none of the resulting pyjama stripes is perfectly horizontal.  Then the intersection of the real axis with each pyjama stripe is $\{t \in \mathbb{R}: tv \in [-\varepsilon,\varepsilon] \pmod{1}\}$ for some nonzero real $v$.  Let $v_1, \ldots, v_N$ denote the $v$'s so obtained.  The union of the sets $\{t \in \mathbb{R}: tv_i \in [-\varepsilon,\varepsilon] \pmod{1}\}$ covers all of $\mathbb{R}$ (since the union of the pyjama stripes covers the whole plane), so $\varepsilon \geq \ML(v_1, \ldots, v_N) \geq \beta(N)$.  It follows that $\alpha(N) \geq \beta(N)$.

An interesting consequence of $\alpha(N)=\beta(N)$ is that this quantity lies within a constant factor $2$ of the trivial volume lower bound $1/(2N)$.  Thus, in order to improve the lower bound for the Pyjama Problem by more than a factor of $2$, one must be able to distinguish between the original formulation and the variant that allows dilations.

\section*{Acknowledgments}
We thank Daniel Zhu for his suggestion of the construction in Lemma~\ref{lem:irrationality-trick}, which allowed us save an exponential in the main theorem, and Noga Alon, Elon Lindenstrauss, Freddie Manners, and Terry Tao for helpful discussions. NK was supported in part by the NSF Graduate Research Fellowship Program under grant DGE–203965. JL was supported by a UCLA dissertation year fellowship and the NSF MSPRF Grant No. 2502827 while this research was conducted.

\bibliography{library}
\bibliographystyle{plain}

\end{document}